\documentclass[12pt]{article}
\usepackage{graphicx,amssymb,xcolor}
\usepackage[bookmarks=true]{hyperref}
\newcommand{\dsp}{\displaystyle}
\newcommand{\bd}{\begin{displaymath}}
\newcommand{\be}{\begin{equation}}
\newcommand{\beq}{\begin{eqnarray}}
\newcommand{\ba}{\begin{array}}
\newcommand{\ed}{\end{displaymath}}
\newcommand{\ee}{\end{equation}}
\newcommand{\eeq}{\end{eqnarray}}
\newcommand{\ea}{\end{array}}
\newcommand{\espace}{\mbox{ }}

\newcommand{\eps}{\varepsilon}
\newcommand{\abs}[1]{\left|#1\right|}
\newcommand{\Prob}{{\rm I\hspace{-0.8mm}P}}
\newcommand{\Exp}{{\rm I\hspace{-0.8mm}E}}
\newcommand{\N}{{\mathbb N}}

\newcommand{\Z}{{\mathbb Z}}
\newcommand{\R}{{\mathbb R}}

\newcommand{\dt}{\partial_t}
\newcommand{\dx}{\partial_x}
\newcommand{\eqref}[1]{(\ref{#1})}

\newtheorem{theorem}{Theorem}[section]

\newtheorem{proposition}{Proposition}[section]
\newtheorem{definition}{Definition}[section]
\newtheorem{lemma}{Lemma}[section]
\newtheorem{corollary}{Corollary}[section]

\newtheorem{remark}{Remark}[section]
\newenvironment{proof}[2]{\espace\\{\em Proof of #1 \ref{#2}.}}{\hfill\mbox{$\square$}}
\begin{document}
\title{ Constructive  Euler hydrodynamics for one-dimensional attractive particle systems
}
\author{C. Bahadoran$^{a,e}$, H. Guiol$^{b}$, K. Ravishankar$^{c,f}$, E. Saada$^{d,e}$}
\date{\today}

\maketitle
{\large \sl To Chuck Newman, Friend, Colleague, and Mentor}
$$ \ba{l}
^a\,\mbox{\small Laboratoire de Math\'ematiques Blaise Pascal, Universit\'e Clermont Auvergne, 63177 Aubi\`ere, France} \\
\quad \mbox{\small e-mail:
bahadora@math.univ-bpclermont.fr}\\
^b\, \mbox{\small Universit\'e Grenoble Alpes, CNRS UMR 5525, TIMC-IMAG,} \\
\quad \mbox{\small 
Computational and Mathematical Biology, 38042 Grenoble cedex, France
} \\
\quad \mbox{\small e-mail:
 guiolh@univ-grenoble-alpes.fr
}\\
^c\, \mbox{\small NYU Shanghai and NYU-ECNU Institute of Mathematical Sciences at NYU-Shanghai, 
China
} \\
\quad \mbox{\small e-mail:
 kr26@nyu.edu}\\
^d\, \mbox{\small CNRS, UMR 8145, MAP5,
Universit\'e Paris Descartes,
Sorbonne Paris Cit\'e, France}\\
\quad \mbox{\small e-mail:
Ellen.Saada@mi.parisdescartes.fr}\\ \\
^e\, \mbox{\small  Supported by  grant ANR-15-CE40-0020-02 }\\
^f\, \mbox{\small Supported by Simons Foundation Collaboration grant 281207}\\
\ea
$$

\begin{abstract}
We review a  (constructive)  approach first introduced in \cite{bgrs1} and further 
developed in \cite{bgrs2,bgrs3,mrs, bgrs4} for hydrodynamic limits 
of asymmetric attractive particle systems,  in a weak or in a strong 
(that is, almost sure) sense,
in an homogeneous or in a quenched disordered setting. 
\end{abstract}

\section{Introduction}\label{sec:intro}
 Among the most studied conservative 
interacting particle systems (\emph{IPS})
 are the
simple exclusion and the zero-range processes. They are {\em
attractive processes}, and possess a one-parameter family of
product extremal invariant and translation invariant probability measures,
that we denote by $\left\{ \nu _\alpha \right\}_{\alpha}$, where
$\alpha$ represents the mean density of particles per site: for
simple exclusion $\alpha \in [0,1]$, and for zero-range $\alpha
\in [0,+\infty)$ (see \cite[Chapter VIII]{lig1}, and \cite{andjel}). 
Both belong to a more general class of systems with
similar properties, called {\em misanthropes} processes (\cite{coc}).\\ \\
 Hydrodynamic limit (\cite{dp, spohn, kl}) is a law of large
numbers for the time evolution (usually described by a limiting PDE,
 called the hydrodynamic equation) of empirical density
fields in interacting particle systems. 
Most usual IPS can be divided into two groups, diffusive and
hyperbolic. In the first group,  which contains for instance the
symmetric or mean-zero asymmetric simple exclusion process, the
macroscopic$\to$microscopic space-time scaling is
$(x,t)\mapsto(Nx,N^2t)$ with $N\to\infty$, and the limiting PDE is a
diffusive equation. In the second group, which contains for instance
the nonzero mean asymmetric simple exclusion process, the scaling is
$(x,t)\mapsto(Nx,Nt)$, and the limiting PDE is of Euler type.
In
both groups this PDE  often  exhibits nonlinearity, either via the
diffusion coefficient in the first group, or via the flux function
in the second one. This raises special difficulties in the
hyperbolic case, due to shocks and non-uniqueness for the solution
of the PDE, in which case the natural problem is to establish
convergence to the so-called \textit{entropy solution} (\cite{serre}). \\ \\
In most known
results, only a weak law of large numbers is established. In this
description one need not have an explicit construction of the
dynamics: the limit is shown in probability with respect to the law
of the process, which is characterized in an abstract way by its
Markov generator and Hille-Yosida's theorem  (\cite{lig1}).
Nevertheless, when simulating particle systems, one naturally uses a
pathwise construction of the process on a
 Poisson space-time random graph (the so-called
\textit{graphical construction}). In this description the dynamics is
deterministically driven by a random space-time measure which
tells when and where the configuration has a chance of being
modified. It is of special interest to show that the hydrodynamic
limit holds for almost every realization of the space-time
measure, as this means  a single simulation is enough to
approximate solutions of the limiting PDE. \\ \\
We are interested  here  in the hydrodynamic behavior of a class of
asymmetric particle systems of $\Z$, which arise as a natural
generalization of the asymmetric exclusion process. 
For such processes, hydrodynamic limit is given by the entropy solutions to a scalar
conservation law of the form
\be \label{hydro_sep}
\partial_t u(x,t)+\partial_x G(u(x,t))=0
\ee 
where $u(.,.)$ is the density field and $G$ - the {\em macroscopic
flux}. The latter is given for the asymmetric exclusion process 
by $G(u)=\gamma u(1-u)$, where $\gamma$ is the mean drift of a particle. 
Because there is a single conserved quantity
({\em i.e.} mass) for the particle system, and an ergodic
equilibrium measure for each density value, \eqref{hydro_sep} can be
guessed through heuristic arguments if one takes for granted that
the system is in {\em local equilibrium}. The macroscopic flux $G$
is obtained by an equilibrium expectation of a microscopic flux
which can be written down explicitly from the dynamics. 
A rigorous proof of the hydrodynamic limit turns
out to be a difficult problem, mainly because of the non-existence
of strong solutions for \eqref{hydro_sep} and the non-uniqueness of
weak solutions. Since the conservation law is not sufficient to
pick a single solution, the so-called entropy weak solution must
be characterized by additional properties; one must then look for
related properties of the particle system to establish its
convergence to the entropy solution. \\ \\
The derivation of hyperbolic equations  of the form \eqref{hydro_sep}  as hydrodynamic limits
began with the seminal paper \cite{rost}, which established a strong
law of large numbers for the totally asymmetric simple exclusion
process on $\Z$, starting with $1$'s to the left of the origin and
$0$'s to the right. This result was extended by \cite{bf} and
\cite{afs} to nonzero mean exclusion process starting from product
Bernoulli distributions with arbitrary densities $\lambda$ to the
left and $\rho$ to the right (the so-called  \textit{Riemann initial
condition}). The Bernoulli distribution at time $0$ is   related to
the fact that uniform Bernoulli measures are invariant for the
process. For the one-dimensional totally asymmetric  (nearest-neighbor)  $K$-exclusion
process, a particular misanthropes process without explicit invariant measures,  
a strong hydrodynamic limit was established in \cite{timo}, starting
from arbitrary initial profiles, by means of the so-called
\textit{variational coupling},  that is a microscopic version of the Lax-Hopf formula. 
These  were the only strong laws available before the series of works reviewed here. 
A common feature of these works is the use of subadditive
ergodic theorem to exhibit some a.s. limit, 
which is then identified by additional arguments. \\ \\
On the other hand, many {\em weak} laws of large numbers were
established for attractive particle systems. A first series of
results treated  systems with product invariant measures and product
initial distributions. In \cite{ak}, for a particular zero-range
model, a weak law was deduced from conservation of local equilibrium
under Riemann initial condition. It was then extended in \cite{av}
to the misanthrope's process of \cite{coc} under an additional
convexity assumption on the flux function. These were substantially
generalized (using Kru\v{z}kov's entropy inequalities,  see \cite{k}) 
in \cite{fraydoun} to multidimensional attractive
systems with product invariant measures
for arbitrary Cauchy data, without any convexity
requirement on the flux.
 In  \cite{rez}, 
 using an abstract 
characterization of the evolution semigroup associated with the limiting equation, 
hydrodynamic limit was
established for the one-dimensional nearest-neighbor $K$-exclusion
process. \\ \\
The above results are concerned with translation-invariant particle dynamics. We are also interested 
in hydrodynamic limits of particle systems in random environment,
leading to homogeneization effects, where an effective diffusion matrix or
flux function is expected to capture the effect of inhomogeneity. 
Hydrodynamic limit in random environment has been widely addressed and robust methods  
have been  developed in the diffusive case.  
In the hyperbolic setting, 
  the few available results in random environment (prior to \cite{bgrs4})
depended on particular features of the investigated models. 
In \cite{bfl}, the authors prove, for the
asymmetric zero-range process with site disorder on $\Z^d$,   
  a quenched hydrodynamic limit given by a hyperbolic
conservation law with an effective homogeneized flux function. To
this end, they use in particular the existence of explicit product
invariant measures for the disordered zero-range process  below
some critical value of the disorder parameter.   In \cite{ks}, extension to the supercritical case
is carried out in the totally asymmetric case with constant jump rate. 
In \cite{timo},  the author establishes a  quenched hydrodynamic
limit for  the totally asymmetric nearest-neighbor $K$-exclusion
process on $\Z$ with  i.i.d  site disorder,  for which explicit
invariant measures are not known. The  last  two results rely on  variational coupling.   However,
the simple exclusion process beyond the totally asymmetric nearest-neighbor case, or more complex
models with state-dependent jump rates, remain outside the scope of  this approach. \\ \\
In this paper, we review successive stages  (\cite{bgrs1, bgrs2, bgrs3, mrs, bgrs4}) 
of a {\em constructive} approach to hydrodynamic limits  given by equations 
of the type \eqref{hydro_sep}, 
which ultimately led us in \cite{bgrs4} to a very general hydrodynamic limit 
result 
for attractive particle systems in one dimension in ergodic random environment.
 We shall  detail our method 
in the setting of \cite{bgrs4}. 
However, we will first explain our approach and advances along the
progression of papers, and quote  results for each one, since they are 
interesting in their own. We hope this could be helpful for a reader 
looking for hydrodynamics of a specific model: according to the available 
knowledge on this model, this reader
could derive either 
a weak, or a strong (without disorder or with a quenched disorder) hydrodynamic limit.\\ \\
 Our  motivation for \cite{bgrs1} was to prove with a constructive method
hydrodynamics of one-dimensional attractive dynamics with product invariant measures, 
but without a concave/convex flux, in view of examples of $k$-step exclusion processes 
and misanthropes processes.
 We initiated for that a ``resurrection''
of the approach of \cite{av}, and we introduced a variational formula
for the entropy solution of the hydrodynamic equation in the Riemann case, and an
approximation scheme to go from a Riemann to a general initial profile. 
Our method is based on an interplay of macroscopic properties for the conservation law 
and analogous microscopic properties for the particle system.
 The next stage, 
achieved in \cite{bgrs2}, was to derive  hydrodynamics 
(which was still a weak law)
for attractive processes without explicit invariant measures. In the same setting, 
we then obtained almost sure hydrodynamics in \cite{bgrs3}, 
 relying on a graphical representation of the dynamics. 
The latter result, apart from its own interest, proved to be an essential step
to obtain quenched hydrodynamics in the disordered case (\cite{bgrs4}). 
For this last paper,
we also relied on \cite{mrs}, which improves an essential property we use, macroscopic stability.\\ \\
Let us mention that we are now working (\cite{bmrs3, bmrs4}) 
 on the hydrodynamic behavior of the disordered asymmetric zero-range process. 
This model falls outside the scope of the present paper because it exhibits  
a phase transition with a critical density above which no invariant measure exists. 
In the supercritical regime, the hydrodynamic limit cannot be established by local 
equilibrium arguments, and condensation may occur locally on a finer scale than 
the hydrodynamic one. Other issues related to this model have been studied recently 
in \cite{bmrs1,bmrs2}.   \\ \\
This review paper
 is organized as follows. 
In Section \ref{sec:notation}, after giving general notation and definitions, 
we  introduce the two basic models we originally worked with,  the misanthropes process
and the $k$-step exclusion process. Then 
we describe informally the  results, and the  main ideas involved in  
\cite{av}  which was our starting point, and in each of the papers
\cite{bgrs1,bgrs2,bgrs3,mrs,bgrs4}. 
Section \ref{sec:results} contains our main results, stated (for convenience) for 
 the misanthropes process.  We then aim at explaining how these results are proved. 
In Section \ref{sec:pde}, we first give a self-contained introduction to scalar conservation 
laws, with the main definitions and results important for our purposes; then we explain 
the derivation of our variational formula in the Riemann case
(illustrated by an example of 2-step exclusion process), and finally our
approximation scheme to solve the general Cauchy problem. 
In Section \ref{proof_hydro}, we outline the most important steps of 
our proof of hydrodynamic limit in a quenched disordered setting: again, 
we first deal with the Riemann problem, then with the Cauchy  problem. 
Finally, in Section 
\ref{sec:othermodels}, we define a general framework which enables 
to describe a class of models
possessing the necessary properties to derive hydrodynamics, and study
a few  examples. 
\section{Notation and preliminaries}\label{sec:notation}
Throughout this paper $\N=\{1,2,...\}$ will denote the set of
natural numbers, and $\Z^+=\{0,1,2,...\}$ the set of non-negative
integers,  and $\R^{+*}=\R^+\setminus\{0\}$ the
set of positive real numbers. 
The integer part  $\lfloor x\rfloor\in\Z$ of $x\in\R$    is
uniquely defined by $\lfloor x\rfloor \leq x<\lfloor x\rfloor +1$.\\ \\
The set of environments  (or disorder)  
is a probability space $({\mathbf A},{\mathcal F}_{\mathbf A},Q)$, 
where $\mathbf A$ is a  compact
metric space and ${\mathcal F}_{\mathbf A}$ its Borel
$\sigma$-field. On $\mathbf A$ we have a group of space shifts
$(\tau_x:\,x\in\Z)$, with respect to which $Q$ is ergodic. \\ \\
We consider particle  configurations  (denoted by greek letters
$\eta,\xi\ldots$) 
 on $\Z$  with at most $K$ (but always finitely many) particles 
 per site, for some given $K\in\N\cup\{+\infty\}$.
Thus the state space, which will be denoted by $\mathbf X$, is either $\N^\Z$ 
in the case $K=+\infty$, or $\{0,\cdots,K\}^\Z$ for $K\in\N$.
For $x\in\Z$ and $\eta\in{\mathbf X}$, $\eta(x)$ denotes the number of particles on site $x$. 
This state space is endowed with the product
topology, which makes it a  metrisable space, 
compact when $\mathbf X=\{0,\cdots,K\}^\Z$.   \\ \\
 A function $f$ defined on ${\mathbf A}\times{\mathbf X}$
 (resp. $g$   on ${\mathbf A}\times{\mathbf X}^2$,  $h$ on $\mathbf X$) 
 is called {\em local} if there is a finite subset
$\Lambda$ of $\Z$ such that $f(\alpha,\eta)$  depends only on
$\alpha$ and $(\eta(x),x\in\Lambda)$
 (resp. $g(\alpha,\eta,\xi)$ depends only on
$\alpha$ and $(\eta(x),\xi(x), x\in\Lambda)$, 
 $h(\eta)$ depends only on
 $(\eta(x), x\in\Lambda)$). 
We denote  again  by $\tau_x$ either the spatial translation operator on the
real line for $x\in\R$, defined by $\tau_x y=x+y$, or its
restriction to  $x\in\Z$. By extension, if $f$ is a function defined
on $\Z$ (resp. $\R$), we set $\tau_x f=f\circ\tau_x$ for $x\in\Z$
(resp. $\R$). In the sequel this will be applied to particle
configurations $\eta\in\mathbf X$, disorder configurations
$\alpha\in\mathbf{A}$, or joint disorder-particle configurations
$(\alpha,\eta)\in\mathbf{A}\times\mathbf{X}$. In the latter case,
unless mentioned explicitely, $\tau_x$  
 applies simultaneously to both components.\\ \\
If $\tau_x$ acts on some set and $\mu$ is a measure on this set,
$\tau_x\mu=\mu\circ\tau_{x}^{-1}$.
We let ${\mathcal M}^+(\R)$ denote the set of  nonnegative
measures on $\R$ equipped with the metrizable topology of vague
convergence, defined by convergence on continuous test functions
with compact support. The set of probability measures on
$\mathbf{X}$ is denoted by ${\mathcal P}(\mathbf{X})$. If $\eta$ is
an ${\mathbf X}$-valued random variable and $\nu\in{\mathcal
P}(\mathbf{X})$, we write $\eta\sim\nu$ to specify that $\eta$ has
distribution $\nu$.  Similarly, for
$\alpha\in\mathbf{A}, Q\in{\mathcal P}(\mathbf{A})$, $\alpha\sim Q$
means that $\alpha$ has distribution $Q$. 
\subsection{Preliminary definitions}\label{subsec:basic_def}
Let us introduce briefly the various notions we shall use in this review,
in view of the next section, where we informally tell the content of each of our papers.
We shall be more precise in the following sections.
Reference books are \cite{lig1, kl}. \\ \\
\textbf{The process.} 
We work with a conservative (i.e. involving only particle jumps but no creation/annihilation),
attractive  (see Definition \eqref{attractive_2} below)  Feller process 
$(\eta_t)_{t\ge 0}$ of state space $\mathbf X$.
When this process evolves in a random environment $\alpha\in\mathbf{A}$,
we denote its  generator by  $L_\alpha$ and its semigroup by $(S_\alpha(t),t\ge 0)$.
Otherwise we denote them by  $L$ and  $S(t)$. 
 In the absence of disorder,  we denote by $\mathcal S$ 
the set of translation invariant probability measures on $\mathbf X$,
 by  $\mathcal I$ 
the set of  invariant probability measures for the process $(\eta_t)_{t\ge 0}$,
and by  $({\mathcal I}\cap{\mathcal S})_e$ 
the set of extremal invariant and translation invariant
probability measures for $(\eta_t)_{t\ge 0}$.  In the disordered case, 
$\mathcal S$ will denote the set of translation invariant probability measures 
on ${\bf A}\times{\bf X}$, see Proposition \ref{invariant}.  \\ \\
 A sequence  $(\nu_n,n\in\N)$ of probability measures on
${\mathbf X}$ converges weakly to some $\nu\in{\mathcal
P}(\mathbf{X})$, if and only if  $\lim_{n\to\infty}\int
f\,d\nu_n=\int f\,d\nu$  for every continuous function $f$ on
${\mathbf X}$. The topology of weak convergence is metrizable and
makes ${\mathcal P}({\mathbf X})$ compact  when $\mathbf X=\{0,\cdots,K\}^\Z$. \\ \\
We equip $\mathbf X$ with the \textit{coordinatewise order}, defined for 
$\eta,\xi\in\mathbf X$ by $\eta\leq\xi$ if and only if $\eta(x)\leq\xi(x)$ for all $x\in\Z$.
A partial stochastic order is defined on ${\mathcal P}(\mathbf{X})$;
namely, for $\mu_1,\mu_2\in{\mathcal P}(\mathbf{X})$,
we write $\mu_1\leq\mu_2$ if the following equivalent conditions
hold (see {\em e.g.}  \cite{lig1, strassen}):\\
  \textit{(i)} For
every non-decreasing nonnegative function $f$ on $\mathbf X$, $\int
f\,d\mu_1\leq\int f\,d\mu_2$.\\
 \textit{(ii)} There exists a coupling
measure  $\overline{\mu}$ on $
\mathbf {X}\times\mathbf {X}$ with marginals $\mu_1$ and $\mu_2$,
such that $\overline{\mu}\{(\eta,\xi):\,\eta\leq\xi\}=1$.\\ \\
The process $(\eta_t)_{t\ge 0}$ is \textit{attractive} 
if its semigroup acts monotonically on probability measures, that is: 
for any $\mu_1,\mu_2\in{\mathcal P}(\mathbf{X})$,
\be \label{attractive_2} \mu_1\leq\mu_2\Rightarrow\forall
t\in\R^+,\,\mu_1 S_\alpha(t)\leq\mu_2 S_\alpha(t) \ee \\
\textbf{Hydrodynamic limits.}
Let $N\in\N$ be the \textit{scaling parameter} for the hydrodynamic limit,
that is, the inverse of the macroscopic distance between two
consecutive sites. The empirical measure of a configuration $\eta$
viewed on scale $N$ is given by 
\[
\pi^N(\eta)(dx)=N^{-1}\sum_{y\in\Z}\eta(y)\delta_{y/N}(dx)
\in{\mathcal M}^+(\R) \]
 where, for $x\in\R$, $\delta_x$ denotes the Dirac measure at $x$, 
  and ${\mathcal M}^+(\R)$ denotes the space of Radon measures on $\R$. 
 This space will be endowed with the metrizable  topology of vague convergence, 
 defined by convergence against the set $C^0_K(\R)$ of continuous test functions 
 on $\R$ with compact support. Let $d_v$ be a distance associated with this topology,
  and $\pi_.$, $\pi'_.$ be two mappings from $[0,+\infty)$ to ${\mathcal M}^+(\R)$. 
  We set
\begin{eqnarray*}
D_T(\pi_.,\pi'_.) & := & \mbox{\rm ess}\sup_{t\in[0,T]}d_v(\pi_t,\pi'_t)\\
D(\pi_.,\pi'_.) & := & \sum_{n=0}^{+\infty}2^{-n}\min[1,D_n(\pi_.,\pi'_.)]
\end{eqnarray*}
A sequence $(\pi^n_.)_{n\in\N}$ of random ${\mathcal M}^+(\R)$-valued paths 
is said to converge locally uniformly in probability to a random 
${\mathcal M}^+(\R)$-valued path $\pi_.$ if, for every $\varepsilon>0$,
\[
\lim_{n\to+\infty}\mu^n\left(
D(\pi^n_.,\pi_.)>\varepsilon
\right)=0
\]
where $\mu^n$ denotes the law of $\pi^n_.$.  \\ \\
 Let us now recall, in the context of scalar conservation laws, 
 standard definitions in hydrodynamic limit theory.
Recall that $K\in\Z^+\cup\{+\infty\}$ 
 bounds the 
number of particles per site.
Macroscopically, the set of possible particle densities will be $[0,K]\cap\R$.
Let $G:[0,K]\cap\R\to\R$ be a Lipschitz-continuous function, called the
{\em flux}.  It is a.e. differentiable, and its derivative $G'$ is
an (essentially) uniformly bounded function. We consider the scalar conservation law 
\be \label{hydrodyn} \dt u+\dx[G(u)]=0 \ee  
where $u=u(x,t)$ is some $[0,K]\cap\R$-valued density field defined on
$\R\times\R^+$. We denote by $L^{\infty,K}(\R)$ the set of bounded 
Borel functions from $\R$ to $[0,K]\cap\R$.
 \begin{definition}\label{def:density-profile}
Let $(\eta^N)_{N\geq 0}$ be a sequence of  $\mathbf X$-valued random variables, 
 and $u_0\in L^{\infty,K}(\R)$.  We say that the sequence $(\eta^N)_{N\geq 0}$  has:\\ \\
(i) {\rm weak density profile}
$u_0(.)$,  if $\pi^N(\eta^N)\to u_0(.)$ in probability with respect to 
the topology of vague convergence, that is equivalent to: 
for all $\varepsilon>0$ and test function $\psi\in C^0_K(\R)$, 
\[
\lim_{N\to\infty} \mu^N\left(\left|\int_{\R}\psi(x)\pi^N(\eta^N)(dx)
-\int_{\R} \psi(x)u_0(x)dx\right|>\varepsilon\right)=0
\]
where $\mu^N$ denotes the law of $\eta^N$.\\ \\
(ii)  {\rm strong density profile} $u_0(.)$,  if the random variables 
are defined on a common probability space $(\Omega_0,{\mathcal F}_0,\Prob_0)$, 
and $\pi^N(\eta^N)\to u_0(.)$ $\Prob_0$-almost surely with respect to the 
topology of vague convergence, that is equivalent to: 
for all test function $\psi\in C^0_K(\R)$, 
\[
 \Prob_0\left(\lim_{N\to\infty}\int_{\R}\psi(x)\pi^N(\eta^N)(dx)
=\int_{\R} \psi(x)u_0(x)dx\right)=1,
\]
\end{definition}
 We consider hydrodynamic limits under hyperbolic time scaling, that is
$Nt$, since we work with asymmetric dynamics. 
 \begin{definition}\label{def:hydro-lim}
  The sequence $(\eta^N_t,\,t\geq 0)_{N\geq 0}$ has
 {\rm hydrodynamic
limit} (resp. {\rm a.s. hydrodynamic limit})
$u(.,.)$  if: for all $t\geq 0,\,(\eta^N_{Nt})_N$ has weak (resp. strong)
density profile $u(.,t)$  where $u(.,t)$ is the {\rm weak entropic solution} of
 \eqref{hydrodyn} with initial condition $u_0(.)$, 
for an appropriately defined {\rm macroscopic flux function} $G$,
where $u_0$ is the density profile of the sequence $(\eta^N_0)_{N}$ 
in the sense of Definition \ref{def:density-profile}. 
\end{definition}
\subsection{Our motivations and approach}\label{sec:story}
Most results on hydrodynamics deal with dynamics with product invariant measures;
 in the most familiar cases, the flux function appearing in the hydrodynamic equation is 
convex/concave
(\cite{kl}). But for many usual examples, the first or the second statement
is not true.\\ \\
\textbf{Reference examples.}
 We present the attractive misanthropes process on one hand,
and the $k$-step exclusion process on the other hand: 
these two  classical  examples will illustrate our purposes along this review. 
In  these basic examples, we take $\alpha\in{\mathbf A}=[c,1/c]^\Z$  
(for a constant $0<c<1$) as the space of environments; 
this  corresponds to site disorder. We also consider those models
 without disorder, which corresponds to $\alpha(x)\equiv 1$.
However, 
our approach applies to a much broader class of models and environments,  as
will be explained in Section \ref{sec:othermodels}.
\\ \\
\textit{The misanthrope's process} 
was introduced in \cite{coc} (without disorder). It has state space 
either $\mathbf X=\N^\Z$ or $\mathbf X=\{0,\cdots,K\}^\Z$ ($K\in\N$), 
and $b\ :\Z^+\times\Z^+\to \R^+$ is the jump rate function. 
A particle present on $x\in\Z$ chooses $y\in\Z$
with probability  $p(y-x)$, where $p(.)$ (the particles' jump kernel)
 is an  asymmetric  probability measure on $\Z$,  and jumps to $y$  at rate 
$\alpha(x)b(\eta(x),\eta(y))$.  
 We assume the following:\label{assumptions_M}\\ \\
{\em (M1)} $b(0,.)=0$, with a $K$-exclusion rule when $\mathbf X=\{0,\cdots,K\}^\Z$: $b(.,K)=0$;\\ 
{\em (M2)} Attractiveness: $b$ is nondecreasing (nonincreasing) in its first (second) argument.\\ 
{\em (M3)} $b$ is a bounded function.\\
 {\em (M4)}  $p$ has a finite first moment, that is, $\sum_{z\in\Z}\abs{z}p(z)<+\infty$.  \\ \\
The quenched disordered process has generator
\be \label{generator} L_\alpha f(\eta)=\sum_{x,y\in{\Z}}\alpha(x)
p(y-x)b(\eta(x),\eta(y)) \left[ f\left(\eta^{x,y} \right)-f(\eta)
\right] \ee
where $\eta^{x,y}$ denotes
the new state after a particle has jumped from $x$ to $y$ (that is
$\eta^{x,y}(x)=\eta(x)-1,\,\eta^{x,y}(y)=\eta(y)+1,\,
\eta^{x,y}(z)=\eta(z)$ otherwise). \\
There are two well-known particular cases of attractive misanthropes processes:
the \textit{simple exclusion process}
(\cite{lig1}) corresponds to 
\[\mathbf X=\{0,1\}^\Z \quad\mbox{ with }\quad b(\eta(x),\eta(y))=\eta(x)(1-\eta(y));
\]
the \textit{zero-range process} (\cite{andjel}) corresponds to 
\[\mathbf X=\N^\Z \quad\mbox{ with }\quad b(\eta(x),\eta(y))=g(\eta(x)),
\]
for a non-decreasing function $g\ :\Z^+\to \R^+$ (not necessarily bounded).\\
 Let us now restrict ourselves to the model without disorder. 
For the simple exclusion and zero-range processes,  $({\mathcal I}\cap{\mathcal S})_e$ is 
a one-parameter family of  product 
probability measures.
The flux function is convex/concave for simple exclusion, but not necessarily
for zero-range.
However, in the general set-up of misanthropes processes,
unless the rate function $b$ satisfies additional algebraic conditions (see \cite{coc,fgs}),
the model does not have product invariant measures;
Even when this is the case, the flux function is not necessarily convex/concave. 
 We refer the reader to 
\cite{bgrs1,fgs}  for examples of misanthropes processes with product invariant measures. 
Note also that a misanthropes process with product invariant measures 
generally loses this property if disorder is introduced, 
with the sole known exception of the zero-range process (\cite{bfl,ev}). \\ \\
The \textit{$k$-step exclusion process} ($k\in\N$)
was introduced in \cite{guiol} (without disorder). 
Its state space is ${\mathbf X}:=\{0,1\}^{\Z}$, and $\{X_n\}_{n\in\N}$ is
 a Markov chain on $\Z$ with
transition matrix $p(.,.)$, distribution ${\bf P}^x$
(and expectation ${\bf E}^x$)
when its initial state is $x\in\Z$. It has generator
\begin{eqnarray}\label{generator_k} 
L_\alpha f(\eta)=\sum_{x,y\in\Z}\alpha(x)c(x,y,\eta)\left[
f\left(\eta^{x,y}\right)-f(\eta) \right]
 \qquad\mbox{with}\\ \nonumber
c(x,y,\eta )=\eta(x)(1-\eta(y)){\bf E}^x\left[
\prod_{i=1}^{\sigma _y-1}
\eta (X_i),\sigma _y\leq\sigma_x,\sigma _y\leq k\right]
\end{eqnarray}
where $\sigma_y=\inf\left\{ n\geq 1:X_n=y\right\} $ is the first
(non zero) arrival time to site $y$ of the chain starting at site
$x$. In words if a particle at site $x$ wants to jump it may go
to the first empty site encountered before returning to site $x$
following the chain $X_n$ (starting at $x$) provided it takes
less than $k$ attempts; otherwise the movement is cancelled.\\
When $k=1$, we recover the simple exclusion
process.\\
The $k$-step exclusion is an attractive process.\\ 
 Let us now restrict ourselves to the model without disorder.  
Then $({\mathcal I}\cap{\mathcal S})_e$ is
a one-parameter family of product Bernoulli measures.\\ 
 In
the totally asymmetric  nearest-neighbor  case, $c(x,y,\eta)=1$ if
$\eta(x)=1$, $y-x\in\{1,\ldots,k\}$ and $y$ is the first
nonoccupied site to the right of $x$; otherwise $c(x,y,\eta)=0$.
The flux function belongs to ${\mathcal C}^2(\R)$, 
it has one inflexion point, 
thus it is neither convex nor concave.  Besides,
flux functions with arbitrarily many inflexion points can be constructed 
by superposition of different $k$-step exclusion processes with 
different kernels and different values of $k$ (\cite{bgrs1}).  \\ \\
\textbf{A constructive approach to hydrodynamics.}
 To overcome the difficulties to derive hydrodynamics
raised by the above examples, 
our starting point was the constructive approach introduced in \cite{av}.
There, the authors proved the \textit{conservation of local equilibrium} 
for the one-dimensional zero-range process with a concave macroscopic flux function $G$
 \textit{in the Riemann case} (in a translation invariant setting, 
 $G$ is the mean flux of particles through the origin),  that is
\be\label{loc_eq}
\forall t>0,\quad\eta^{N}_{Nt}
\stackrel{{\mathcal L}}{\to}\nu_{u(t,x)}
\ee
where $\nu_\rho$ is the product invariant measure of the zero-range process 
with mean density $\rho$, and $u(.,.)$ is  
 the \textit{entropy solution} of the conservation law 
\be \label{hydrodynamics-av} \dt u+\dx[G(u)]=0;
\quad u(x,0)= R_{\lambda,\rho}(x)= \lambda\mathbf{1}_{\{x<0\}}+ \rho\mathbf{1}_{\{x\ge 0\}}\ee
 One can show (see \cite{kl}) that \eqref{loc_eq} implies 
the hydrodynamic limit in the sense of Definition \ref{def:hydro-lim}. 
Let us begin by explaining (informally) their method.
They first show in \cite[Lemma 3.1]{av} 
that a weak Cesaro limit of (the measure of)
the process is an invariant and translation invariant measure, thus a convex combination
of elements of $({\mathcal I}\cap{\mathcal S})_e$, the one-parameter family of
extremal invariant and translation invariant 
probability measures for the dynamics. Then they compute 
in \cite[Lemma 3.2]{av} the
(Cesaro) limiting density inside a macroscopic box,
 thanks to the explicit knowledge of the product measures
elements  of $({\mathcal I}\cap{\mathcal S})_e$. 
They prove next in \cite[Lemma 3.3 and Theorem 2.10]{av} 
that the above convex combination is in fact the Dirac measure concentrated on 
the solution of the hydrodynamic equation, thanks to the concavity of
their flux function. They conclude by  proving 
that the Cesaro limit implies the weak limit via monotonicity arguments, in 
\cite[Propositions 3.4 and 3.5]{av}.
Their proof is valid for misanthrope processes with product invariant measures and a concave
macroscopic flux.\\ \\
In \cite{bgrs1}, we derive by a constructive method the hydrodynamic behavior 
 of 
attractive processes with finite range
irreducible jumps, and for which the set $({\mathcal I}\cap{\mathcal S})_e$
consists in a one-parameter family of explicit product measures 
 but the flux is not necessarily convex or concave. 
Our approach relies on (i) an explicit construction
of Riemann solutions without assuming convexity of the macroscopic flux, 
 and (ii) a general result which proves that the hydrodynamic limit for Riemann initial profiles implies
the same for general initial profiles. \\
For point (i),  we
rely on the (parts of) the proofs in \cite{av} based only on attractiveness and on the
knowledge of the product measures composing
 $({\mathcal I}\cap{\mathcal S})_e$, 
and we provide a new approach otherwise. Instead of the convexity assumption on the flux,
which belongs here to ${\mathcal C}^2(\R)$,
we prove that the solution of the hydrodynamic equation is given by
a variational formula, whose index set is an interval, namely the set of values of the parameter
of the elements of $({\mathcal I}\cap{\mathcal S})_e$. 
Knowing  $({\mathcal I}\cap{\mathcal S})_e$ explicitly enables us  to deal with dynamics 
with the non compact state space $\N^\Z$.\\
 Point (ii) 
is based on an approximation scheme inspired by Glimm's scheme 
for hyperbolic systems of conservation laws (see \cite{serre}). 
Among our tools are the \textit{finite propagation property} 
and the \textit{macroscopic stability} of the
dynamics. The latter property is due to \cite{bm};  
both  require finite range transitions.\\
We illustrate our results on  variations of our above reference examples. \\ \\
 While the results and exemples of \cite{bgrs1} include the case $K=+\infty$, 
 in our subsequent works,
for reasons explained below, we considered $K<+\infty$, 
thus $\mathbf X=\{0,\cdots,K\}^\Z$, which will be assumed from now on. Under this additional assumption, 
in \cite{bgrs2}, we extend the hydrodynamics result of \cite{bgrs1}
to dynamics without explicit invariant measures.
Indeed, thanks to monotonicity, we prove that $({\mathcal I}\cap{\mathcal S})_e$
is still a one-parameter family of probability measures, for which the set $\mathcal R$ of values of
the parameter 
is a priori not an interval anymore, but a closed subset  of $[0,K]$: 
\begin{proposition}\label{prop_inv_bgrs2} (\cite[Proposition 3.1]{bgrs2}).
Assume $p$ satisfies the irreducibility assumption
\be\label{irreducibility_misanthrope}
\forall z\in\Z,\quad\sum_{n=1}^{+\infty}\left[
p^{*n}(z)+p^{*n}(-z)
\right]>0
\ee
where $p^{*n}$ denotes the $n$-th convolution power of the kernel $p$, 
that is the law of the sum of $n$ independent $p$-distributed random variables.
Then there exists a closed subset $\mathcal R$ of $[0,K]$, containing $0$ and $K$, such that
\be\label{inv_bgrs2}
({\mathcal I}\cap{\mathcal S})_e=\{\nu^\rho:\,\rho\in{\mathcal R}\}
\ee
where the probability measures $\nu^\rho$ on $\bf X$ satisfy the following properties:
\be\label{ergodic_bgrs2}
\lim_{l\to+\infty}(2l+1)^{-1}\sum_{x=-l}^l\eta(x)=\rho,\quad\nu^\rho\mbox{- a.s.}
\ee
and
\be\label{order_bgrs2}
\rho\leq\rho'\Rightarrow\nu^\rho\leq\nu^{\rho'}
\ee
\end{proposition}
Following the same general scheme as in \cite{bgrs1}, but with additional difficulties, 
we then obtain the following main result.
\begin{theorem}
\label{th_hydro_bgrs2}
(\cite[Theorem 2.2]{bgrs2}).
Assume $p(.)$ 
satisfies the irreducibility assumption \eqref{irreducibility_misanthrope}.
Then there exists a Lipschitz-continuous function $G:[0,K]\to\R^+$ 
such that the following holds. Let $u_0\in L^{\infty,K}(\R)$, and $(\eta^N_.)_N$ 
be any sequence of processes with generator \eqref{generator}, 
such that the sequence $(\eta^N_0)_N$ has density profile $u_0(.)$. Then, 
the sequence $(\eta^N_{N.})_N$ has hydrodynamic limit given by $u(.,.)$,  the entropy solution to \eqref{hydrodyn} with initial condition $u_0(.)$.
\end{theorem}
The drawback is that we have (and it will also be the case in the following papers) to 
restrict ourselves to dynamics with compact state space  
to prove hydrodynamics with general initial data.  
This is necessary to define the macroscopic flux outside $\mathcal R$, 
 by a linear interpolation; this makes this flux 
 Lipschitz continuous, a minimal requirement to define entropy solutions. 
 We have to consider  a
$\mathcal R$-valued Riemann problem, for which   
we prove conservation of local equilibrium. 
Then we use an averaging argument to prove hydrodynamics
  (in the absence of product invariant measures, 
 the passage from local equilibrium to hydrodynamics  is no longer a consequence of  \cite{kl}). 
For  general initial profiles, we have to refine the approximation procedure
of \cite{bgrs1}: we go first to $\mathcal R$-valued entropy solutions,  then to
arbitrary entropy solutions.\\ \\
In \cite{bgrs3}, by a refinement of our method,
we obtain a  \textit{strong}  (that is an almost sure) hydrodynamic limit, when starting
from an arbitrary initial profile.
 By almost sure, we mean that we construct the process with generator 
\eqref{generator} on an explicit probability space defined as the product 
$(\Omega_0\times\Omega,{\mathcal F}_0\otimes{\mathcal F},\Prob_0\otimes\Prob)$, 
where $(\Omega_0,{\mathcal F}_0,\Prob_0)$ is a probability space used to construct 
random initial states, and $(\Omega,{\mathcal F},\Prob)$ is a Poisson space used 
to construct the evolution from a given state.
\begin{theorem}\label{th_hydro_bgrs3}(\cite[Theorem 2.1]{bgrs3})
Assume $p(.)$ has finite first moment  and satisfies the irreducibility assumption
\eqref{irreducibility_misanthrope}. Then
the following holds, where $G$ is the same function as in Theorem \ref{th_hydro_bgrs2}.
Let $(\eta^N_0,\,N\in\N)$ be any sequence of ${\mathbf X}$-valued random
variables on a probability space $(\Omega_0,\mathcal F_0,\Prob_0)$
such that
\be\label{initial_profile_vague}
\lim_{N\to\infty}\pi^N(\eta^N_0)(dx)=
u_0(.)dx\quad\Prob_0\mbox{-a.s.}\ee
for some measurable $[0,K]$-valued profile $u_0(.)$.i
Then  the
$\Prob_0\otimes\Prob$-a.s. convergence 
\be \label{later_profile_strong}
\lim_{N\to\infty}\pi^N(\eta^N_{Nt})(dx)=u(.,t)dx
\ee
 holds uniformly on all bounded time intervals, where $(x,t)\mapsto
u(x,t)$ denotes the unique  entropy solution to \eqref{hydrodyn}  
with initial condition $u_0(.)$.
\end{theorem}
Our constructive approach requires new ideas since the sub-additive ergodic
theorem (central to the few previous existing proofs for strong hydrodynamics) 
is no longer effective in our setting. 
We work with the graphical
representation of the dynamics, on which we couple an arbitrary number of processes,
thanks to the  \textit{complete monotonicity property}  of the dynamics.
To solve the 
$\mathcal R$-valued Riemann problem, we combine proofs of almost sure analogues
of the results of \cite{av}, rephrased for currents which become our centerpiece, 
with a space-time ergodic theorem for particle systems
and large deviation results for the empirical measure.
In the approximation steps, new error analysis is necessary: In particular, 
we have to do an explicit time discretization
(vs. the ``instantaneous limit'' of \cite{bgrs1,bgrs2}), 
we need estimates uniform in time, and each approximation step
requires a control with exponential bounds.
\\ \\
 In \cite{mrs} we derive the macroscopic stability property 
when the particles' jump kernel $p(.)$ has a finite first moment and a positive mean.
We also extend under those hypotheses the ergodic theorem for densities 
due to \cite{rez} that we use in \cite{bgrs3}. Finally, we prove the 
finite propagation property when $p(.)$ has a finite third moment.
This enables us to get rid of the finite range assumption on $p$
required so far, and to extend the strong hydrodynamic result
of \cite{bgrs3} when the particles' jump kernel has a finite third moment
 and a positive mean. \\ \\
In \cite{bgrs4}, we  derive, thanks to the tools introduced in \cite{bgrs3},
a quenched strong hydrodynamic limit for a bounded 
attractive particle system on $\Z$  evolving in a random ergodic environment.
 (This result, which contains Theorems \ref{th_hydro_bgrs2} and 
\ref{th_hydro_bgrs3} above, is stated later on in this paper as Theorem \ref{th:hydro}).
 Our method is robust with respect to
the model and disorder (we are not restricted to site or
bond disorder).
We introduce a general framework  to describe the rates
of the dynamics, which applies to a large class of models.
To overcome the difficulty of the
simultaneous loss of translation invariance and  lack of
knowledge of explicit invariant measures for the disordered system,
we study  a joint
disorder-particle process, which is translation invariant.
We characterize its extremal invariant and translation invariant measures,
and prove its strong hydrodynamic limit.
This implies the quenched hydrodynamic result we look for.\\
We illustrate our results on various examples. 
\section{Main results}\label{sec:results}
 The construction of interacting particle systems is done either 
analytically, through generators and semi-groups (we refer to \cite{lig1}
for systems with compact state space, and to \cite{ls,andjel,fgs}
otherwise), or through a graphical representation. Whereas 
the former is sufficient
to derive hydrodynamic limits in a weak sense,  which  is done in \cite{bgrs1,bgrs2},
the latter is necessary to derive strong hydrodynamic limits, 
which  is done in \cite{bgrs3,bgrs4}. First, we  explain
in Subsection \ref{subsec:graphical} the graphical construction,
then in Subsection \ref{subsec:hydro-inv} we detail our results 
from  \cite{bgrs4} 
on invariant measures for the dynamics and hydrodynamic limits. \\ 
For simplicity, we restrict ourselves in this section 
to the misanthropes process
with site disorder,
 which corresponds to the generator \eqref{generator}. 
However, considering only the necessary properties of the misanthropes
process required to prove our hydrodynamic results, 
it is possible to deal with more general models 
including the $k$-step exclusion process,
by embedding them in a global framework, 
in which the dynamics is viewed as a random transformation of the configuration; 
the latter simultaneously defines the graphical construction {\em and} generator.
 More general forms of random environments than site disorder can also be considered. 
In Subsection \ref{subsec:required} we list the above required
properties of misanthropes processes, and we defer the 
study of the $k$-step exclusion process 
and more general models to Section \ref{sec:othermodels}. 
\subsection{Graphical construction}\label{subsec:graphical}
 This subsection is based on \cite[Section 2.1]{bgrs3}. 
We now describe the graphical construction
 (that is the pathwise construction on a Poisson space) of the  system given by
\eqref{generator}, which uses a Harris-like representation  (\cite{h,har}; 
see for instance  
\cite{afs, frigray, bmrs2, swart}  for details and
justifications). This enables us to define the evolution from
arbitrarily many different initial configurations simultaneously on the same
probability space, in a way that depends monotonically on these initial configurations. \\ \\ 
 We consider the probability space $(\Omega,{\mathcal
F},\Prob)$ of measures $\omega$ on $\R^+\times\Z^2\times[0,1]$ of
the form
\[
\omega(dt,dx,dz,du)=\sum_{m\in\N}\delta_{(t_m,x_m,z_m,u_m)}
\]
where  $\delta_{(\cdot)}$ denotes Dirac measure, and
$(t_m,x_m,z_m,u_m)_{m\ge 0}$ are pairwise distinct and form a
locally finite set. The $\sigma$-field $\mathcal F$ is generated
by the mappings $\omega\mapsto\omega(S)$ for Borel sets $S$. The
probability measure $\Prob$ on $\Omega$ is the one that makes
$\omega$ a Poisson process with intensity
\[
m(dt,dx,dz,du)=
||b||_\infty\lambda_{\R^+}(dt)\times\lambda_{\Z}(dx)\times
p(dz)\times\lambda_{[0,1]}(du) \]
where $\lambda$ denotes either the Lebesgue or the counting
measure.
We denote by $\Exp$ the corresponding expectation. Thanks to
assumption  {\em (M4)} on page \pageref{assumptions_M}, 
we can proceed as in \cite{afs, frigray} (for a construction with a weaker assumption
we refer to \cite{bmrs2, swart}):  for $\Prob$-a.e. $\omega$, there exists a
unique mapping
\be \label{unique_mapping_0} (\alpha,\eta_0,t)\in{\mathbf A}\times{\bf
X}\times\R^+\mapsto\eta_t= \eta_t(\alpha,\eta_0,\omega) \in{\mathbf X} \ee
satisfying: (a) $t\mapsto \eta_t(\alpha,\eta_0,\omega)$ is
right-continuous; (b) $\eta_0(\alpha,\eta_0,\omega)=\eta_0$;  (c) for
$t\in\R^+$, $(x,z)\in\Z^2$, $\eta_t=\eta_{t^-}^{x,x+z}$ if
\begin{equation}
\label{update} \exists u\in[0,1]:\,\omega\{(t,x,z,u)\}=1\mbox{ and
} u\leq\displaystyle{\alpha(x)\frac{b(\eta_{t^-}(x),\eta_{t^-}(x+z))}
{||b||_\infty}}
\end{equation}
and (d) for all $s,t\in\R^{+*}$ and $x\in\Z$,
\begin{equation}
\label{noupdate} \omega\{[s,t]\times Z_x\times (0,1)\}=0 \Rightarrow
\forall v\in[s,t],\eta_v(x)=\eta_s(x)
\end{equation}
where
\[
Z_x:=\left\{(y,z)\in\Z^2: y=x \mbox{ or } y+z=x \right\}
\]
In short, \eqref{update} tells how the state of the system can be
modified by an ``$\omega$-event'', and \eqref{noupdate} says that
the system cannot be modified outside $\omega$-events.\\ \\
Thanks to assumption \textit{(M2)}  on page \pageref{assumptions_M},   we have that
\be \label{attractive_1}
(\alpha,\eta_0,t)\mapsto\eta_t(\alpha,\eta_0,\omega) \mbox{ is
nondecreasing w.r.t. } \eta_0\ee
 Property \eqref{attractive_1} 
implies \eqref{attractive_2}, that is, attractiveness.
 But it is more powerful: it  implies the {\em complete monotonicity} 
property (\cite{fm, dplm}),
that is, existence of a monotone Markov coupling for an \textit{arbitrary} number 
of processes with generator  \eqref{generator},  which is necessary in our proof 
of strong hydrodynamics for
general initial profiles.  
The coupled process can be defined by a Markov generator, 
as in \cite{coc} for two components,  that is
\begin{eqnarray}\label{def:gen-coupl}
&&\overline{L}_\alpha f(\eta,\xi)
=\cr
&  &\sum_{x,y\in\Z:\,x\neq y}
 \Bigl\{\alpha(x)p(y-x)[b(\eta(x),\eta(y))\wedge b(\xi(x),\xi(y))]
 \left[f(\eta^{x,y},\xi^{x,y})-f(\eta,\xi)\right]\cr
&& +  \alpha(x)p(y-x)[b(\eta(x),\eta(y))-b(\xi(x),\xi(y))]^+\left[f(\eta^{x,y},\xi)-f(\eta,\xi)\right]\cr
&& +   \alpha(x)p(y-x)[b(\xi(x),\xi(y))- b(\eta(x),\eta(y))]^+
\left[f(\eta,\xi^{x,y})-f(\eta,\xi)\right]\Bigr\}
\end{eqnarray} 
One may further introduce an ``initial'' probability space
$(\Omega_0,{\mathcal F}_0,\Prob_0)$, large enough to construct
random initial configurations $\eta_0=\eta_0(\omega_0)$ for
$\omega_0\in\Omega_0$.
The general process with random initial configurations is
constructed on the enlarged space
$(\widetilde{\Omega}=\Omega_0\times\Omega,\widetilde{\mathcal
F}=\sigma({\mathcal F}_0\times{\mathcal
F}),\widetilde{\Prob}=\Prob_0\otimes\Prob)$ by setting
\[
\eta_t(\alpha,\widetilde{\omega})=\eta_t(\alpha,\eta_0(\omega_0),\omega)\]
for $\widetilde{\omega}=(\omega_0,\omega)\in\widetilde{\Omega}$.
One can show (see  for instance \cite{bmrs2, frigray, swart})  that this defines
a Feller process with generator \eqref{generator}:
 that is for any $t\in\R^+$ and $f\in C(\mathbf{X})$
(the set of continuous functions on $\mathbf{X}$),
 $S_\alpha(t)f\in C(\mathbf{X})$ where
$S_\alpha(t)f(\eta_0)=\Exp[f(\eta_t(\alpha,\eta_0,\omega))]$.
If $\eta_0$ has distribution $\mu_0$, then the process thus
constructed is Feller with generator \eqref{generator} and initial
distribution $\mu_0$.\\ \\
We define on $\Omega$ the {\sl space-time shift} $\theta_{x_0,t_0}$:
for any $\omega\in\Omega$, for any $(t,x,z,u)$
\be\label{def:spacetimeshift}
(t,x,z,u)\in\theta_{x_0,t_0}\omega\mbox{ if and only if
}(t_0+t,x_0+x,z,u)\in\omega\ee
 where $(t,x,z,u)\in\omega$ means $\omega\{(t,x,z,u)\}=1$. By its very definition, 
 the mapping introduced in
\eqref{unique_mapping_0} enjoys the following properties, for all
 $s,t\geq 0$, $x\in\Z$ and 
 $(\eta,\omega)\in{\mathbf X}\times{\Omega}$:
\be \label{mapping_markov}
\eta_s(\alpha,\eta_t(\alpha,\eta,\omega),\theta_{0,t}\omega)=\eta_{t+s}(\alpha,\eta,\omega)
\ee
which  implies Markov property, and
\be \label{mapping_shift}
\tau_x\eta_t(\alpha,\eta,\omega)=\eta_t(\tau_x\alpha,\tau_x\eta,\theta_{x,0}\omega)
\ee
which  yields the \textit{commutation property}
\be\label{commutation}
L_\alpha\tau_x=\tau_x L_{\tau_x\alpha}
\ee
so that  $S_\alpha(t)$  and $\tau_x$ commute.
\subsection{Hydrodynamic limit and invariant measures}\label{subsec:hydro-inv}
 This section is based on \cite[Sections 2, 3]{bgrs4}. 
A central issue in interacting particle systems, 
and more generally in the theory of Markov processes, 
is the characterization of invariant measures (\cite{lig1}). 
Besides, this characterization plays a crucial role in the 
derivation of hydrodynamic limits (\cite{kl}).
 We detail here our results on these two questions.\\ \\
\textbf{Hydrodynamic limit.}
 We first state the strong hydrodynamic behavior of the process 
 with quenched  site  disorder. 
\begin{theorem}\label{th:hydro}  (\cite[Theorem 2.1]{bgrs4}). 
Assume  $K<+\infty$,  $p(.)$ has finite third moment, and satisfies the irreducibility assumption
 \eqref{irreducibility_misanthrope}. 
 Let $Q$ be an ergodic probability distribution on $\mathbf A$. Then
there exists a Lipschitz-continuous function  $G^Q$ on
$[0,K]$ defined  in  \eqref{flux} and \eqref{inter_1}--\eqref{inter_2}  
 below (depending only on  $p(.)$, $b(.,.)$
   and $Q$)
such that the following holds.
Let $(\eta^N_0,\,N\in\N)$ be a sequence of ${\mathbf X}$-valued random
variables on a probability space $(\Omega_0,\mathcal F_0,\Prob_0)$
such that
\be\label{initial_profile_vague}
\lim_{N\to\infty}\pi^N(\eta^N_0)(dx)=
u_0(.)dx\quad\Prob_0\mbox{-a.s.}\ee
for some measurable $[0,K]$-valued profile $u_0(.)$.
Then for  $Q$-a.e.  $\alpha\in{\mathbf A}$, the
$\Prob_0\otimes\Prob$-a.s. convergence
\be \label{later_profile_strong}
\lim_{N\to\infty}\pi^N(\eta_{Nt}(\alpha,\eta^N_0(\omega_0),\omega))(dx)=u(.,t)dx
\ee
 holds uniformly on all bounded time intervals, where $(x,t)\mapsto
u(x,t)$ denotes the unique entropy solution with initial condition
$u_0$ to the conservation law
\be \label{hydrodynamics} \dt u+\dx[G^Q(u)]=0 \ee
\end{theorem}
 The  $\Prob$-almost sure convergence in Theorem \ref{th:hydro} 
refers to the graphical construction of Subsection \ref{subsec:graphical}, 
and is stronger than the usual notion of hydrodynamic limit, which is a 
convergence in probability  (cf. Definition \ref{def:hydro-lim}). 
The strong hydrodynamic limit implies the weak one (\cite{bgrs3}). 
This leads to the following weaker but more usual statement, 
which also has the advantage of not depending on a particular construction of the process.
\begin{theorem}\label{th:hydro_weak}
Under assumptions and notations of Theorem \ref{th:hydro}, 
there exists a subset ${\mathbf A}'$ of $\mathbf A$, with $Q$-probability $1$, 
such that the following holds for every $\alpha\in{\mathbf A}'$.
For any $u_0\in L^\infty(\R)$, and any sequence $\eta^{N}_.=(\eta^N_t)_{t\geq 0}$ 
of processes with generator \eqref{generator} satisfying the convergence in probability

\be\label{initial_profile_weak}
\lim_{N\to\infty}\pi^N(\eta^N_0)(dx)=
u_0(.)dx,\ee
one has the locally uniform convergence in probability
\be \label{later_profile}
\lim_{N\to\infty}\pi^N\left(\eta^N_{Nt}\right)(dx)=u(.,t)dx
\ee
\end{theorem}
\begin{remark}\label{remark_bgrs2}
Theorem \ref{th_hydro_bgrs2} (resp. Theorem \ref{th_hydro_bgrs3}) 
is a special case of Theorem \ref{th:hydro_weak} (resp. Theorem \ref{th:hydro}). 
Indeed, it suffices to consider the ``disorder'' distribution $Q$ that is 
the Dirac measure supported on the single environment $\alpha(x)\equiv 1$.
\end{remark}
Note that the statement of Theorem \ref{th:hydro_weak} is stronger than 
hydrodynamic limit  in the sense of Definition \ref{def:hydro-lim}, 
because it states convergence of the empirical measure {\em process} 
rather then convergence at every fixed time. 
 To define the {\sl macroscopic flux} $G^Q$, we first  define
 the {\sl microscopic flux}  as follows. The generator 
  \eqref{generator}  can be decomposed as a sum of translates
of a ``seed'' generator centered around the origin:
\be\label{decomp_gengen}
L_\alpha=\sum_{x\in\Z}L^x_\alpha
\ee
 Note that such a decomposition is not unique. A natural choice of
 component $L_\alpha^x$ at $x\in\Z$ is given in the case of \eqref{generator} by 
\be\label{shifted_component}
L_\alpha^x f(\eta)=\alpha(x)\sum_{z\in\Z}p(z)\left[f\left(\eta^{x,x+z}\right)-f(\eta)\right]
\ee 
We then define $j$ to be either of the functions $j_1$, $j_2$ defined below:
\be\label{def_f}
j_1(\alpha,\eta):=L_\alpha\left[\sum_{x>0}\eta(x)\right], \quad
j_2(\alpha,\eta):=L_\alpha^0\left[
\sum_{x\in\Z}x\eta(x)
\right]
\ee
The definition of $j_1$ is partly formal, because the function 
$\sum_{x>0}\eta(x)$ does not belong to the domain of the generator
$L_\alpha$. Nevertheless, the formal computation gives rise to a well-defined 
function $j_1$, because  the rate $b$ is a local function. 
Rigorously, one defines $j_1$ by difference, as the unique function such that
\be\label{j1_bis}
j_1-\tau_x j_1= L_\alpha  \left[\sum_{y=1}^x\eta(y)\right]
\ee
for every $x\in\N$. The action of the generator in \eqref{j1_bis} is now well defined, 
because we have a local function that belongs to its domain.\\ \\
 In the case of \eqref{generator}, we obtain the following microscopic flux functions:
\begin{eqnarray}\nonumber
j_1(\alpha,\eta) & = & \sum_{(x,z)\in\Z^2,x\leq 0<x+z}\alpha(x)b(\eta(x),\eta(x+z))\\
 \label{microflux_misanthrope_1}
& - & \sum_{(x,z)\in\Z^2,x+z\leq 0<x}\alpha(x)b(\eta(x),\eta(x+z))\\
\label{microflux_misanthrope_2}
j_2(\alpha,\eta) & = & \alpha(0)\sum_{z\in\Z} zp(z)b(\eta(0),\eta(z))
\end{eqnarray}
Once a microscopic flux function $j$ is defined, the macroscopic flux function 
is obtained by averaging with respect to a suitable family of measures,
that we now introduce.\\ \\
\textbf{Invariant measures.}
We define the 
markovian  {\em joint disorder-particle
process} $(\alpha_t,\eta_t)_{t\ge 0}$ on ${\mathbf A}\times {\mathbf
X}$ with generator given by, for any local function $f$ on ${\mathbf
A}\times{\mathbf X}$, 
\be \label{generator_joint-general} \mathfrak L 
f(\alpha,\eta)=
\sum_{x,y\in{\Z}}\alpha(x)
p(y-x)b(\eta(x),\eta(y)) \left[ f\left(\alpha,\eta^{x,y} \right)-f(\alpha,\eta)
\right] 
\ee
Given
$\alpha_0=\alpha$, this dynamics simply means that $\alpha_t=\alpha$
for all $t\geq 0$, while $(\eta_t)_{t\ge 0}$  is a Markov process with
generator $L_\alpha$ given by \eqref{generator}. Note that $ \mathfrak L $ is
\textit{translation invariant}, that is 
\be\label{eq:L-transl-inv}\tau_x  \mathfrak L = \mathfrak L \tau_x\ee
 where $\tau_x$
acts jointly on $(\alpha,\eta)$. This is equivalent to the commutation relation \eqref{commutation}
for the quenched dynamics. 
\\ \\
Let $\mathcal I_{ \mathfrak L }$,
$\mathcal S$ and ${\mathcal S}^{\mathbf A}$ denote the sets of
probability measures  that are respectively  invariant for $ \mathfrak L $,
shift-invariant
 on $\mathbf{A}\times\mathbf{X}$ and
shift-invariant  on $\mathbf A$.
\begin{proposition}\label{invariant} (\cite[Proposition 3.1]{bgrs4}). 
For every $Q\in{\mathcal S}^{\mathbf A}_e$, there exists a closed
subset $\mathcal R^Q$ of $[0,K]$ containing $0$ and $K$, 
depending on  $p(.)$ and $b(.,.)$,  such that
\begin{eqnarray*}
\left(\mathcal I_{ \mathfrak L }\cap\mathcal S\right)_e & = &
\left\{\nu^{Q,\rho},\,Q\in{\mathcal S}^{\mathbf
A}_e,\,\rho\in{\mathcal R}^Q\right\}
\end{eqnarray*}
where index $e$ denotes the set of extremal elements, and
$(\nu^{Q,\rho}:\,\rho\in\mathcal R^Q)$ is a family of
shift-invariant measures on ${\mathbf A}\times{\mathbf X}$, weakly
continuous with respect to $\rho$, such that
\beq\label{densite-rho}
\int\eta(0)\nu^{Q,\rho}(d\alpha,d\eta)&=&\rho\\
\label{Rezakhanlou}
\lim_{l\to\infty}(2l+1)^{-1}\sum_{x\in\Z:|x|\le l}\eta(x)&=&\rho,
\quad \nu^{Q,\rho}-\mbox{a.s.} \\
\label{ordered_measures}\rho\leq\rho'&\Rightarrow&\nu^{Q,\rho}\ll\nu^{Q,\rho'}\eeq
 Here, $\ll$ denotes the conditional stochastic order defined in Lemma \ref{lemma_conditional} below.
\end{proposition}
 {}From the family of invariant measures in the above  proposition  
for the joint disorder-particle process, one may deduce a family 
of invariant measures for the quenched particle process.
\begin{corollary}
\label{corollary_invariant} (\cite[Corollary 3.1]{bgrs4}). 
There exists  a
subset $\mathbf{\widetilde{A}}^Q$ of ${\mathbf A}$ with
$Q$-probability $1$ (depending on  $p(.)$ and $b(.,.)$, 
such that the family of probability measures
$(\nu^{Q,\rho}_\alpha:\,\alpha\in\mathbf{\widetilde{A}}^Q,\rho\in\mathcal
R^Q)$ on ${\mathbf X}$, defined by
$\nu^{Q,\rho}_\alpha(.):=\nu^{Q,\rho}(.|\alpha)$  
satisfies the following properties, for every $\rho\in\mathcal R^Q$:\\ \\
%
%
\noindent{(B1)} For every  $\alpha\in\mathbf{\widetilde{A}}^Q$, $\nu^{Q,\rho}_\alpha$
is an
invariant measure for $L_\alpha$.\\ \\
{(B2)} For every  $\alpha\in\mathbf{\widetilde{A}}^Q$, $\nu^{Q,\rho}_\alpha$-a.s.,
\[\lim_{l\to\infty}(2l+1)^{-1}\sum_{x\in\Z:\,|x|\leq
l}\eta(x)=\rho\]
{(B3)} The quantity
\be \label{flux} G^Q_\alpha(\rho):=\int j(\alpha,\eta)\nu^{Q,\rho}_\alpha(d\eta)=:G^Q(\rho),\quad
\rho\in{\mathcal R}^Q \ee
does not depend on  $\alpha\in\mathbf{\widetilde{A}}^Q$.
\end{corollary}
\begin{remark}\label{remark_bgrs2_inv}
As already observed in Remark \ref{remark_bgrs2} above,  
we can view the non-disordered model as a special ``disordered'' model 
by taking $Q$ to be the Dirac measure on the constant environment 
$\alpha(x)\equiv 1$. Hence, Proposition \ref{prop_inv_bgrs2} is a special case
 of Proposition \ref{invariant}. Note that we then have 
 $\nu^\rho=\nu^{Q,\rho}=\nu^{Q,\rho}_\alpha$ for $Q$-a.e. $\alpha\in\bf A$.
\end{remark}
 We now come back to the macroscopic flux function $G^Q(\rho)$. 
We define it as \eqref{flux} for $\rho\in\mathcal R^Q$ and
we extend it by linear interpolation on the complement of $\mathcal R^Q$,
which is a finite or countably infinite union of disjoint open
intervals:  that is, we set
\be\label{inter_1}
G^Q(\rho):=\frac{\rho-\rho^-}{\rho^+-\rho^-}G(\rho^+)+
\frac{\rho^+-\rho}{\rho^+-\rho^-}G(\rho^-),\quad
\rho\not\in{\mathcal R}^Q
\ee
where
\be\label{inter_2}
\rho^-:=\sup[0,\rho]\cap{\mathcal R}^Q,\quad
\rho^+:=\inf[\rho,+\infty)\cap{\mathcal R}^Q
\ee
By definition of $\nu_\alpha^{Q,\rho}$ and statement
\textit{(B3)} of Corollary \ref{corollary_invariant}, we also have
\be\label{flux_joint}
  G^Q(\rho):=\int j(\alpha,\eta)\nu^{Q,\rho}(d\alpha,d\eta) ,\quad
\rho\in{\mathcal R}^Q
\ee
We point out that \eqref{flux_joint} yields the same macroscopic flux function, 
whether $j=j_1$ or $j=j_2$ defined in \eqref{def_f} is plugged into it. 
It does not depend on the choice of a particular decomposition 
\eqref{decomp_gengen} either. These invariance properties follow from 
translation invariance of $\nu^{Q,\rho}$ and translation invariance 
\eqref{eq:L-transl-inv} of the joint dynamics.  Definitions \eqref{def_f} 
of the microscopic flux, and \eqref{flux}--\eqref{flux_joint} of the macroscopic 
flux are model-independent, and can thus be used for other models, such as the 
$k$-exclusion process, or the models reviewed in Section \ref{sec:othermodels}. 
Of course, the invariant measures involved in \eqref{flux}--\eqref{flux_joint} 
depend on the model and disorder.  \\ \\
The function $G^Q$ can be shown to be Lipschitz continuous  (\cite[Remark 3.3]{bgrs4}). 
 This is  the minimum regularity required
for the classical theory of entropy solutions, 
see Section \ref{sec:pde}. We cannot say more about $G^Q$ in general, because
the measures $\nu^{Q,\rho}_\alpha$ are most often not explicit.\\ \\
 Note that in the special case of the non-disordered model 
investigated in \cite{bgrs2} (see  Proposition \ref{prop_inv_bgrs2}, 
Theorem \ref{th_hydro_bgrs2}, Remarks \ref{remark_bgrs2} and \ref{remark_bgrs2_inv} above), 
the microscopic flux functions do not depend on $\alpha$, and \eqref{flux} 
or \eqref{flux_joint} both reduce to
\be\label{flux_homogeneous}
G(\rho):=\int j(\eta)d\nu^\rho(\eta),\quad
\rho\in{\mathcal R}
\ee 
Even in the absence of disorder, only a few models have explicit invariant measures, 
and thus an explicit flux function.  In Section \ref{sec:othermodels}, 
we define a variant of the $k$-step exclusion process, that we call the 
exclusion process with overtaking. 
 It is possible to tune the microscopic parameters of this model so as to obtain any prescribed 
polynomial flux function (constrained to vanish at density values $0$ and $1$ 
due to the exclusion rule).\\
In contrast, a most natural and seemingly simple generalization of the asymmetric 
exclusion process,  the asymmetric $K$-exclusion process,  for which $K\ge 2$
and $b(n,m)={\bf 1}_{\{n>0\}}{\bf 1}_{\{m<K\}}$ in \eqref{generator}, 
does not have explicit 
invariant measures. Thus, nothing more than the Lipschitz property can be said 
about its flux function in general.  In the special case of the totally 
asymmetric $K$-exclusion process,
that is for  
$p(1)=1$, the flux function 
is shown to be concave in \cite{timo}, as a consequence of the 
variational approach used  there  to derive hydrodynamic limit. 
But this approach does not apply to the models we
consider in the present paper. \\ \\
An important open question is 
whether the set $\mathcal R^Q$  (or its analogue $\mathcal R$ in the absence of disorder) 
covers the whole range of possible densities, 
or if it contains gaps corresponding to phase transitions.
The only partial answer to this question so far was given by the 
following result from \cite{bgrs2} for the totally asymmetric $K$-exclusion process
 without disorder. 
\begin{theorem}\label{RforKexclusion}(\cite[Corollary 2.1]{bgrs2}).
For the totally asymmetric $K$-exclusion process  without disorder, 
 $0$ and $K$ are limit points of $\mathcal R$, and $\mathcal R$ contains at least
one point  in $[1/3,K-1/3]$.
\end{theorem}
 We end this section with a skeleton of proof for Proposition \ref{invariant}.
 We combine the steps done to prove \cite[Proposition 3.1]{bgrs2}  (without disorder)
and  \cite[Proposition 3.1]{bgrs4}. \\  
\begin{proof}{proposition}{invariant}
The proof has two parts.\\
\textit{Part 1.} It is an extension to the joint particle-disorder process
of a classical scheme in a non-disordered setting due to \cite{lig}, 
which is also the basis for similar results 
in \cite{andjel,coc,guiol,gs}. It relies on couplings.\\
\textit{a)} We first need to couple measures, through the following lemma,
 analogous to Strassen's Theorem (\cite{strassen}). 
\begin{lemma}
\label{lemma_conditional}  (\cite[Lemma 3.1]{bgrs4}). 
For two probability measures $\mu^1$,
$\mu^2$ on ${\mathbf A}\times{\mathbf X}$, the
following properties  (denoted  by $\mu^1\ll\mu^2$) are
equivalent:\\
(i) For every bounded measurable local function $f$ on ${\mathbf
A}\times{\mathbf X}$, such that $f(\alpha,.)$ is nondecreasing for
all $\alpha\in{\mathbf A}$, we have $\int f\,d\mu^1\leq\int
f\,d\mu^2$.\\
(ii) The measures $\mu^1$ and $\mu^2$ have a common
$\alpha$-marginal  $Q$, and
$\mu^1(d\eta|\alpha)\leq\mu^2(d\eta|\alpha)$ for $Q$-a.e.
$\alpha\in{\mathbf A}$.\\
(iii) There exists a coupling measure $\overline{\mu}(d\alpha,d\eta,d\xi)$ supported on
$\{(\alpha,\eta,\xi)\in{\mathbf
A}\times{\mathbf X}^2: \eta\leq\xi \}$ 
under which $(\alpha,\eta)\sim\mu^1$ and $(\alpha,\xi)\sim\mu^2$.
\end{lemma}
\textit{b)} Then, for the dynamics,
we denote by $\overline{\mathfrak L}$ the coupled generator for the joint
process $(\alpha_t,\eta_t,\xi_t)_{t\ge 0}$ on ${\mathbf
A}\times{\mathbf X}^2$ defined by
\be\label{joint_coupling}\overline{\mathfrak L}f(\alpha,\eta,\xi)
=(\overline{L}_\alpha f(\alpha,.))(\eta,\xi)\ee
for any local function $f$ on ${\mathbf A}\times{\mathbf X}^2$,
where $\overline{L}_\alpha$ 
 was defined in \eqref{def:gen-coupl}. 
Given $\alpha_0=\alpha$, this means that $\alpha_t=\alpha$ for all
$t\geq 0$, while $(\eta_t,\xi_t)_{t\ge 0}$ is a Markov process with
generator $\overline{L}_\alpha$.  
%
We denote by  $\overline{\mathcal S}$ the set of probability measures on
${\mathbf A}\times{\mathbf X}^2$ that are invariant by space shift
$\tau_x(\alpha,\eta,\xi)=(\tau_x\alpha,\tau_x\eta,\tau_x\xi)$.
We prove successively (next lemma combines \cite[Lemmas 3.2, 3.4 and Proposition 3.2]{bgrs4}):
 \begin{lemma}\label{lemma_successive}
\textit{(i)} Let $\mu',\mu''\in(\mathcal I_{\mathfrak L}\cap\mathcal S)_e$ with a
common $\alpha$-marginal $Q$. Then there exists
$\overline{\nu}\in\left(\mathcal
I_{\overline{\mathfrak L}}\cap\overline{\mathcal S}\right)_e$
such that the respective marginal distributions of $(\alpha,\eta)$
and $(\alpha,\xi)$ under $\overline{\nu}$ are $\mu'$ and
$\mu''$.\\
\textit{(ii)} Let $\overline{\nu}\in\left(\mathcal
I_{\overline{\mathfrak L}}\cap\overline{\mathcal S}\right)_e$. Then
$\overline{\nu}\{(\alpha,\eta,\xi)\in{\mathbf
A}\times{\mathbf X}^2: \eta\leq\xi \}$ and 
$\overline{\nu}\{(\alpha,\eta,\xi)\in{\mathbf A}\times{\mathbf X}^2: \xi\leq\eta\} $ 
belong to $\{0,1\}$. \\ 
\textit{(iii)} Every $\overline{\nu}\in
\left({\mathcal I}_{\overline{\mathfrak L}}\cap\overline{\mathcal S}\right)_e$ 
is supported on 
$\{(\alpha,\eta,\xi)\in{\mathbf A}\times{\mathbf X}^2: \eta\leq\xi \,\mbox{ or }\,\xi\leq\eta\}$.
\end{lemma}
\textit{c)}
This last point \textit{(iii)} is the core of the proof  of Proposition \ref{invariant}: 
Attractiveness assumption 
ensures that an initially
ordered pair of coupled configurations remains ordered at later
times. 
We say that there is a positive (resp. negative) discrepancy 
between two coupled configurations $\xi,\zeta$ at some site $x$ 
if $\xi(x)>\zeta(x)$ (resp. $\xi(x)<\zeta(x)$). 
Irreducibility assumption  \eqref{irreducibility_misanthrope}  
induces  a stronger property: pairs of  discrepancies  of opposite signs
between two coupled configurations eventually get killed, 
so that the two configurations become ordered.\\ \\
\textit{Part 2.}
We define
\[
\mathcal R^Q:=\left\{\int\eta(0)\nu(d\alpha,d\eta):\,
\nu\in\left(\mathcal I_{\mathfrak L}\cap\mathcal S\right)_e,\,\nu\mbox{ has
}\alpha\mbox{-marginal }Q\right\}
\]
Let $\nu^i\in\left(\mathcal I_{\mathfrak L}\cap\mathcal S\right)_e$  with $\alpha$-marginal $Q$ and
$\rho^i:=\int\eta(0)\nu^i(d\alpha,d\eta)\in{\mathcal R}^Q$ for
$i\in\{1,2\}$. Assume $\rho^1\leq\rho^2$.  Using Lemma
\ref{lemma_conditional},\textit{(iii)}, then Lemma
\ref{lemma_successive},  we obtain  $\nu^1\ll\nu^2$,  that is
\eqref{ordered_measures}.
Existence \eqref{Rezakhanlou} of an asymptotic particle density can be
obtained by a proof analogous to \cite[Lemma 14]{mrs}, where the
space-time ergodic theorem
is applied to the joint disorder-particle process. Then,
closedness of ${\mathcal R}^Q$ is established as in \cite[Proposition 3.1]{bgrs2}:
it uses \eqref{ordered_measures}, \eqref{Rezakhanlou}.
Given the rest of the proposition, the weak continuity statement 
comes from a coupling argument, using  \eqref{ordered_measures}  and Lemma
\ref{lemma_conditional}.
\end{proof} 
\subsection{Required properties of the model}\label{subsec:required}
 For the proof of Theorem \ref{th:hydro} and Proposition \ref{invariant}, 
we have not used the particular
form of $L_\alpha$ in \eqref{generator}, but the following  properties.\\ \\
1) The set of environments is a probability space $({\mathbf
A},{\mathcal F}_{\mathbf A},Q)$, where $\mathbf A$ is a  compact
metric space and ${\mathcal F}_{\mathbf A}$ its Borel
$\sigma$-field. On $\mathbf A$ we have a group of space shifts
$(\tau_x:\,x\in\Z)$, with respect to which $Q$ is ergodic.  For each
$\alpha\in{\mathbf A}$, $L_\alpha$ is the generator of a Feller
process on ${\mathbf X}$ 
that satisfies \eqref{commutation}.  The latter 
should be viewed as the assumption on ``how the disorder enters the dynamics''. It 
is equivalent to  $\mathfrak L$  satisfying \eqref{eq:L-transl-inv}, 
that is  being a translation-invariant generator on ${\mathbf A}\times{\mathbf X}$.\\ \\
2) For $L_\alpha$ we can define a graphical construction  
on a space-time Poisson space $(\Omega,\mathcal F,\Prob)$
satisfying the 
 complete monotonicity property
\eqref{attractive_1}. \\ \\
3) Irreducibility  assumption \eqref{irreducibility_misanthrope},
combined with attractiveness  assumption \textit{(M2)} 
 on page \pageref{assumptions_M},  are responsible for 
Lemma \ref{lemma_successive},\textit{(iii)}. \\ \\
Indeed, given the generator \eqref{generator} 
of the process, there is not a unique graphical construction. The strong 
convergence in Theorem \ref{th:hydro} would hold for {\em any} graphical 
construction satisfying \eqref{attractive_1} {\em plus} the existence of a 
sequence of Poissonian events killing any remaining pair of discrepancies
 of opposite signs (see Part 1,c) of the proof of Proposition \ref{invariant} 
for this last point).  
The latter property follows in the case of the misanthrope's process from 
irreducibility assumption \eqref{irreducibility_misanthrope}. 
\\ \\
In Section \ref{sec:othermodels} we shall therefore introduce a general framework
to consider other models satisfying 1) and 2),  
 with appropriate assumptions replacing \eqref{irreducibility_misanthrope} to 
imply Proposition \ref{invariant}.  
 We refer to Lemma \ref{attractive_kstep} for a statement 
and proof of Property \eqref{attractive_1} in the context 
of a general model, including the $k$-step exclusion process. 
The coupled process linked to this property can be  tedious to 
write  in the usual form of explicit coupling rates for more than two
components, or for complex models. We shall see that 
it can be written in a simple model-independent way using 
the framework of Section \ref{subsec:framework}. \\
\section{Scalar conservation laws and entropy solutions}\label{sec:pde}
 In Subsection \ref{subsec:ent}, we recall the definition and characterizations
of entropy solutions to scalar
conservation laws, which will appear as hydrodynamic limits of the
above models. Then in Subsection \ref{subsec:rie}, we explain our variational
formula for the entropy solution in the Riemann case, 
first when the flux function $G$ is Lipschitz continuous, then when 
$G\in {\mathcal C}^2(\R)$  has a single inflexion
point, so that the entropy solution has a more explicit form. Finally, 
in Subsection \ref{subsec:glimm}, we explain the approximation schemes
to go from a Riemann initial profile to a general initial profile. 
\subsection{Definition and properties of entropy solutions}\label{subsec:ent}
 This section is taken from \cite[Section 2.2]{bgrs2} and \cite[Section 4.1]{bgrs3}. 
For more details, we refer to the
textbooks \cite{gr,serre},  or \cite{bres}. 
Equation \eqref{hydrodyn} has no strong solutions in general:
even starting from a smooth Cauchy datum $u(.,0)=u_0$,
discontinuities (called shocks in this context) appear in finite
time. Therefore it is necessary to consider weak solutions, but
then  uniqueness is lost for the Cauchy problem. To recover
uniqueness,
we need to define {\em entropy solutions}.\\ \\
Let $\phi:[0,K]\cap\R\rightarrow\R$ be a convex function. In the context
of hyperbolic systems, such a function is called an
 {\em entropy}. We define the associated
 {\em entropy flux} $\psi$ on $[0,K]$ as
\[
\psi(u):=\int_0^u\phi'(v)G'(v)dv
\]
 $(\phi,\psi)$ is called an
{\em entropy-flux pair}. A Borel function
$u:\R\times\R^{+*}\to[0,K]$ is called an  {\em entropy solution}
to \eqref{hydrodyn}  if and only if it is
entropy-dissipative, {\em i.e.}
\be \label{dissipated_entropy} \dt\phi(u)+\dx\psi(u)\leq 0 \ee
in the sense of distributions on $\R\times\R^{+*}$ for any
entropy-flux pair $(\phi,\psi)$. Note that, by taking $\phi(u)=\pm
u $ and hence $\psi(u)=\pm G(u)$, we see that an entropy solution
is indeed a weak solution to \eqref{hydrodyn}. This
definition can be motivated by the following points:
 {\em i)} when $G$ and $\phi$ are continuously
differentiable, \eqref{hydrodyn} implies equality in strong
sense in \eqref{dissipated_entropy}  (this follows from the chain
rule for differentiation); {\em ii)} this no longer holds in
general if $u$ is only a weak solution to \eqref{hydrodyn};
{\em iii)} the inequality \eqref{dissipated_entropy} can be seen
as a macroscopic version of the second law of thermodynamics that
selects physically relevant solutions. Indeed, one should think of
the {\em concave} function $h=-\phi$ as a thermodynamic entropy,
and spatial integration of \eqref{dissipated_entropy} shows that
the total thermodynamic entropy may not decrease during the
evolution (this is rigorously true for periodic boundary
conditions, in which case the total entropy is well defined). \\ \\
Kru\v{z}kov proved the following fundamental existence and uniqueness 
 result:
\begin{theorem}
\label{kruzkov} (\cite[Theorem 2 and Theorem 5]{k}).
Let $u_0:\R\to[0,K]$ be a Borel  measurable
initial datum. Then there exists a unique (up to a Lebesgue-null
subset of $\R\times\R^{+*}$) entropy solution $u$ to
\eqref{hydrodyn} subject to the initial condition
\be \label{initial} \lim_{t\to 0^+} u(.,t)=u_0(.)\mbox{ in }
L^1_{\rm loc}(\R) \ee
This solution (has a representative in its
$L^\infty(\R\times\R^{+*})$ equivalence class that) is continuous
as a mapping $t\mapsto u(.,t)$ from $\R^{+*}$ to $L^1_{\rm
loc}(\R)$.
\end{theorem}
 We recall here that a sequence $(u_n,n\in\N)$
of Borel measurable functions on $\R$ is said to converge to $u$
in $L^1_{\rm loc}(\R)$ if and only if
\[
\lim_{n\to\infty} \int_I\abs{u_n(x)-u(x)}dx=0
\]
for every bounded interval $I\subset\R$.
\begin{remark}\label{rk:kruzkov} Kru\v{z}kov's theorems are stated for a
continuously differentiable $G$. However the proof of the
uniqueness result (\cite[Theorem 2]{k}) uses only
Lipschitz continuity.  In the Lipschitz-continuous case, existence
could be derived from Kru\v{z}kov's result by a flux approximation
argument. However a different, self-contained (and
constructive) proof of existence in this case can be found in  \cite[Chapter 6]{bres}.
\end{remark}
The following proposition is a collection of results on entropy
solutions. We first recall the following definition.
Let ${\rm TV}_I$ denote the variation of a
function defined on some closed interval $I=[a,b]\subset\R$, {\em
i.e.}
\[
{\rm TV}_I[u(.)]= \sup_{x_0=a<x_1\cdots<x_n=b} \sum_{i=0}^{n-1}
\abs{ u(x_{i+1})-u(x_i) }
\]
 Let us say that $u=u(.,.)$ defined on
$\R\times\R^{+*}$ has locally bounded space variation if
\be \label{def_bounded} \sup_{t\in J} {\rm TV}_I[u(.,t)]<+\infty
\ee
for every bounded closed space interval $I\subset\R$ and bounded
time interval $J\subset\R^{+*}$.\\ \\
For two measures $\alpha,\beta\in{\mathcal M}^+(\R)$ with compact
support, we define
\be\label{def_delta}\Delta(\alpha,\beta):=\sup_{x\in\R}\abs{
\alpha((-\infty,x])-\beta((-\infty,x])}\ee
When $\alpha$ or $\beta$  
is of the form $u(.)dx$ for $u(.)\in L^\infty(\R)$
 with compact support, we simply
write $u$ in \eqref{def_delta} instead of $u(.)dx$.
 For a sequence  $(\mu_n)_{n\ge 0}$  of measures with uniformly bounded support, 
the following equivalence holds:
\be\label{equiv_delta}
\mu_n\to \mu\mbox{ vaguely if and only if }\lim_{n\to\infty}\Delta(\mu_n,\mu)=0
\ee  
\begin{proposition}
\label{standard}  (\cite[Proposition 4.1]{bgrs3}) 
\mbox{}\\ \\
i) Let $u(.,.)$ be the entropy solution to \eqref{hydrodyn}
with Cauchy datum  $u_0\in L^\infty(\R)$. Then the mapping
$t\mapsto u_t=u(.,t)$ lies in $C^0([0,+\infty),L^1_{\rm loc}(\R))$.\\
\\
ii) If $u_0$ has constant value $c$, then for all $t>0$, $u_t$ has constant value $c$.\\ \\
iii) If $u_0^i(.)$ has finite variation, that is ${\rm
TV}u_0^i(.)<+\infty$, then so does $u^i(.,t)$ for every $t>0$, and
${\rm TV}u^i(.,t)\leq{\rm TV}u_0^i(.)$.\\ \\
iv) Finite propagation property: Assume $u^i(.,.)$ ($i\in\{1,2\}$)
is the entropy solution to \eqref{hydrodyn} with Cauchy data
$u_0^i(.)$. Let 
\be\label{speed_pde}
V=||G'||_\infty:=\sup_\rho\abs{G'(\rho)}
\ee
Then,
for every $x<y$ and $0\leq t<(y-x)/2V$,
\be \label{lonestab}
\int_{x+Vt}^{y-Vt}\left[u^1(z,t)-u^2(z,t)\right]^\pm dz\leq\int_x^y\left[u_0^1(z)-u_0^2(z)\right]^\pm dz
\ee
In particular, assume $u_0^1=u_0^2$ (resp. $u_0^1\leq u_0^2$) on $[a,b]$ 
for some  $a,b\in\R$ such that $a<b$. Then, for all
$t\leq(b-a)/(2V)$, $u_t^1=u_t^2$ (resp. $u_t^1\leq u_t^2$) on $[a+Vt,b-Vt]$.\\ \\
(v) If $\dsp\int_\R u_0^i(z)dz<+\infty$, then
\be \label{deltastab}
\Delta(u^1(.,t),u^2(.,t))\leq\Delta(u^1_0(.),u^2_0(.)) \ee
\end{proposition}
Properties \textit{(o)--(iv)} are standard. Property \textit{(v)} can be deduced from the 
correspondence between entropy solutions of \eqref{hydrodyn} 
and viscosity solutions of the Hamilton-Jacobi equation
\be\label{hj}
\partial_t h(x,t)+G[\partial_x h(x,t)]=0
\ee
Namely, $h$ is a viscosity solution of \eqref{hj} if and only if 
$u=\partial_x h$ is an entropy solution of \eqref{hydrodyn}.
Then \textit{(v)} follows from the monotonicity of the solution semigroup 
for \eqref{hj}.  Properties (iv) and (v) have microscopic analogues 
(respectively Lemma \ref{finite_prop} and Proposition \ref{prop:macro-stab}) 
in the class of particle systems we consider, which play an important role 
in the proof of the hydrodynamic limit, as will be sketched in Section \ref{proof_hydro}. \\ \\
 We next recall a possibly more familiar
definition of entropy solutions based on shock admissibility
conditions, but valid only for solutions with bounded variation.
This point of view selects the relevant weak solutions by
specifying what kind of discontinuities are permitted.
 First, in particular,  the following two conditions are
necessary and sufficient for a piecewise smooth function $u(x,t)$
to be a weak solution to equation (\ref{hydrodyn}) 
with initial condition \eqref{eq:rie} (see \cite{ballou}):\\
1. $u(x,t)$ solves equation (\ref{hydrodyn}) at points of
smoothness.\\
2. If $x(t)$ is a curve of discontinuity of the solution then the
{\sl Rankine-Hugoniot condition}
\be\label{rankine-hugoniot}\partial_t x(t) = S[u^+;u^-]:=\frac{G(u^-)-G(u^+)}{u^--u^+}\ee
 holds along $x(t)$.\\
 Moreover, to ensure 
uniqueness, the
following geometric condition, known as \textit{Ole{\u\i}nik's entropy
condition} (see {\em e.g.} \cite{gr} or \cite{serre}),  is sufficient. 
A discontinuity $(u^-,u^+)$, with
$u^\pm:=u(x\pm 0,t)$, is called an entropy shock, if and only if:\\ 
\be \label{oleinik} \ba{l} \mbox{The chord of the graph of $G$ between $u^-$ and $u^+$ lies:}\\
\mbox{below the graph if $u^-<u^+$, above the graph if
$u^+<u^-$.}\ea\ee
In the above condition, ``below" or``above" is meant in wide
sense, {\em i.e.} does not exclude that the graph and chord
coincide at some points between $u^-$ and $u^+$. In particular,
when $G$ is strictly convex (resp. concave), one recovers the fact
that only (and all) decreasing (resp. increasing) jumps are
admitted  (as detailed in subsection \ref{subsec:rie}
below).  Note that, if the graph of $G$ is linear on some
nontrivial interval, condition \eqref{oleinik} implies that any
increasing or decreasing jump within this interval is an
entropy shock.\\ \\
 Indeed,  condition \eqref{oleinik} can be used to select entropy solutions
among weak solutions.  The following result is a consequence of
\cite{vol}.
\begin{proposition}
\label{volpert}  (\cite[Proposition 2.2]{bgrs2}).  
Let $u$ be a weak solution to
\eqref{hydrodyn} with locally bounded space variation. Then
$u$ is an entropy solution to \eqref{hydrodyn} if and only
if, for a.e. $t>0$, all discontinuities of $u(.,t)$ are entropy
shocks.
\end{proposition}
 One can show that, if the Cauchy datum $u_0$ has
locally bounded variation, the unique entropy solution given by
Theorem \ref{kruzkov} has locally bounded space variation. Hence
Proposition \ref{volpert} extends into an existence and uniqueness
theorem within functions of locally bounded space variation, where
entropy solutions may be defined as weak solutions satisfying
\eqref{oleinik}, without reference to \eqref{dissipated_entropy}. 
\subsection{The Riemann problem}\label{subsec:rie}
 This subsection is based first on \cite[Section 4.1]{bgrs2}, then on 
\cite[Section 2.1]{bgrs1}. 
Of special importance among entropy solutions are the solutions of
the Riemann problem, {\em i.e.} the Cauchy problem for particular
initial data of the form
\begin{equation}\label{eq:rie}
R_{\lambda,\rho}(x)=\lambda \mathbf 1_{\{x<0\}}+\rho \mathbf 1_{\{x\geq 0\}}
\end{equation}
Indeed: {\em(i)} as developed in the sequel of this subsection, 
these solutions can be computed  explicitly and
have a variational representation; {\em(ii)} as will be seen in 
Subsection \ref{subsec:glimm}, one can construct
approximations to the solution of the general Cauchy problem by
using only Riemann solutions. This has inspired our belief that
one could derive general hydrodynamics from Riemann
hydrodynamics.\\ \\
 In connection with Theorem \ref{th:hydro} and Proposition \ref{invariant}, 
it will be important in the sequel to consider flux functions $G$
with possible linear degeneracy on some density intervals. 
Therefore, in the sequel of this section, $\mathcal R$
will denote a closed subset of $[0,K]\cap\R$ such that $G$ 
is affine on each of the countably many disjoint open intervals 
whose union is the complement of $\mathcal R$. Such a subset exists 
(for instance one can take $\mathcal R=[0,K]\cap\R$) and is not necessarily unique.\\ \\
{}From now on we assume $\lambda<\rho$; adapting to the case
$\lambda>\rho$ is straightforward, by replacing in the sequel
lower with upper convex hulls, and minima/minimizers with
maxima/maximizers. Consider $G_c$, the lower convex envelope of
$G$ on $[\lambda,\rho]$. There exists a nondecreasing function
$H_c$ (hence with left/right limits) such that $G_c$ has
left/right hand derivative  denoted by  $H_c(\alpha\pm 0)$ at every $\alpha$.
The function $H_c$ is defined uniquely outside
the at most countable set of non-differentiability points of $G_c$
\begin{equation}\label{eq:theta}
\Theta = \{\alpha \in [\lambda,\rho]: H_c(\alpha- 0)<H_c(\alpha+
0)\}
\end{equation}
As will appear below, the particular choice of $H_c$ on $\Theta$
does not matter. Let $v_*=  v_*(\lambda,\rho) = H_c(\lambda+0)$
and $v^*= v^*(\lambda,\rho)= H_c(\rho-0)$. Since $H_c$ is
nondecreasing, there is a nondecreasing function $h_c$ on
$[v_*,v^*]$ such that, for every $v\in[v_*,v^*]$,
\be\label{def_h_1}
\ba{l}
\alpha<h_c(v)\Rightarrow H_c(\alpha)\leq v\\
\alpha>h_c(v)\Rightarrow H_c(\alpha)\geq v
\ea
\ee
Any such $h_c$ satisfies
\be\label{def_h_2}
\ba{l}
h_c(v-0)=\inf\{ \alpha\in\R:\,H_c(\alpha)\geq v \}=\sup\{
\alpha\in\R:\,H_c(\alpha)<v \}\\
h_c(v+0)=\inf\{ \alpha\in\R:\,H_c(\alpha)> v \}=\sup\{
\alpha\in\R:\,H_c(\alpha)\leq v \}
\ea
\ee
 We have that,  anywhere in \eqref{def_h_2},
$H_c(\alpha)$ may be replaced with $H_c(\alpha\pm 0)$. The
following properties can be derived 
from \eqref{def_h_1} and \eqref{def_h_2}:\\
\\
1. Given $G$, $h_c$ is defined uniquely, and is continuous,
outside the at most countable set
\begin{eqnarray}
\Sigma_{low}(G)& =& \{ v \in [v_*,v^*] : G_c \mbox{ is
differentiable with derivative }v \nonumber\\&& \mbox{ in a
nonempty open subinterval of } [\lambda,\rho]\}
\end{eqnarray}
By
``defined uniquely'' we mean that for such $v$'s, there is a
unique $h_c(v)$ satisfying \eqref{def_h_1}, which does not depend
on the choice of $H_c$ on $\Theta$.\\
\\
2. Given $G$, $h_c(v\pm 0)$ is uniquely defined, {\em i.e.}
independent of the choice of $H_c$ on $\Theta$, for any
$v\in[v_*,v^*]$. For $v\in\Sigma_{low}(G)$, $(h_c(v-0),h_c(v+0))$
is the maximal
open interval over which $H_c$ has constant value $v$.\\
\\
3. For every $\alpha\in\Theta$ and
$v\in(H_c(\alpha-0),H_c(\alpha+0))$, $h_c(v)$ is uniquely defined
and equal to $\alpha$.\\
\\
In the sequel we extend $h_c$ outside $[v_*,v^*]$ in a natural way
by setting
\be \label{extend_h} h_c(v)=\lambda \mbox{ for }v<v_*,\quad
h_c(v)=\rho\mbox{ for }v>v^* \ee \\
Next proposition  extends 
\cite[Proposition 2.1]{bgrs1}, where we assumed $G \in {\mathcal C}^2(\R)$.
\begin{proposition}\label{proposition_1}  (\cite[Proposition 4.1]{bgrs2}). 
Let $\lambda,\rho\in[0,K]\cap\R$, $\lambda<\rho$.  For $v\in\R$, we set
\be\label{limiting_current_0}
\mathcal G_v(\lambda,\rho)  := \inf\left\{
G(r)-vr:\,r\in[\lambda,\rho]\cap\mathcal R
\right\}\label{limiting_current}\ee
Then 
\phantom{ii}i) For every $v\in\R\setminus\Sigma_{low}(G)$,
 The minimum in \eqref{limiting_current_0} is achieved at a 
unique point $h_c(v)$, and $u(x,t):=h_c(x/t)$ is the weak entropy
solution to (\ref{hydrodyn}) with Riemann initial condition
\eqref{eq:rie},  denoted by $R_{\lambda,\rho}(x,t)$.  \\ \\
\phantom{i}ii) If $\lambda\in\mathcal R$ and $\rho\in\mathcal R$, 
the previous minimum is unchanged if  restricted to
$[\lambda,\rho]\cap\mathcal R$. As a result, the Riemann entropy
solution is a.e. $\mathcal R$-valued.\\ \\
 iii) For every $v,w\in\R$,
\be
\int_v^w R_{\lambda,\rho}(x,t)dx  =  t[\mathcal
G_{v/t}(\lambda,\rho)-\mathcal G_{w/t}(\lambda,\rho)]\label{limiting_current}
\ee 
\end{proposition}
 In the case when $G\in {\mathcal C}^2(\R)$ is such
that $G''$ vanishes only finitely many times, the 
expression of $u(x,t)$ is more explicit.
Let us detail,  as in \cite[Section 2.1]{bgrs1},   following  
\cite{ballou}, the case where $G$ has 
a single inflexion
point $a\in (0,1)$ and $G(u)$ is strictly convex in $0 \le u<a$
and strictly concave in $a< u\le 1$. 
We denote $H=G'$, $h_1$ the inverse of $H$ restricted to
$(-\infty,a)$, and  $h_2$  the inverse of $H$ restricted to
$(a,+\infty)$. We have
\begin{lemma}(\cite[Lemma 2.2, Lemma 2.4]{ballou}).\label{Bl2.4}
Let $w<a$ be given, and suppose that $w^*<\infty$. Then
\item a) $S[w;w^*]=H(w^*)$;
\item b) $w^*$ is the only zero of $S[u;w]-H(u)$, $u>w$.
\end{lemma}
If $\rho <\lambda<\rho^*$ $(\rho <a)$, then 
$H(\lambda)>H(\rho)$: the characteristics starting
from $x\leq 0$ have a speed (given by $H$) greater than the
speed of those starting from $x>0$. If the characteristics
intersect along a curve $x(t)$, then Rankine-Hugoniot condition
will be satisfied if
\[
x'(t)=S[u^+;u^-]=\frac{G(\lambda)-G(\rho)}{\lambda-\rho}
=S[\lambda;\rho].
\]
The convexity of $G$ implies that
Ole{\u\i}nik's Condition is satisfied across $x(t)$.
Therefore the unique entropic weak solution is the \textit{shock}:
\begin{equation}\label{shock}
u(x,t)=\left\{
\begin{array}{ll}
\lambda, & x\leq S[\lambda;\rho]t;\\ \rho, & x > S[\lambda;\rho]t;\\
\end{array}
\right.
\end{equation}
If $\lambda<\rho<a$, the relevant part of the flux function is
convex. The characteristics
starting respectively from $x\leq 0$ and $x>0$ never meet,
and they never enter the space-time wedge between lines
$x=H(\lambda) t$ and $x=H(\rho) t$. It is possible to
define piecewise smooth weak solutions with a jump occurring in
the wedge satisfying the Rankine-Hugoniot condition. But the
convexity of $G$ prevents such solutions from satisfying Ole{\u\i}nik's 
Condition. Thus the unique entropic weak solution is the \textit{continuous
solution with a rarefaction fan}:
\begin{equation}\label{fan}
u(x,t)=\left\{
\begin{array}{ll}
\lambda, & x\leq H(\lambda)t;\\ h_1(x/t), & H(\lambda)t<
x \leq H(\rho)t;\\ \rho, & H(\rho)t< x;\\
\end{array}
\right.
\end{equation}
Let $\rho<\rho^*< \lambda$ $(\rho<a)$: Lemma \ref{Bl2.4} applied
to $\rho$ suggests that a jump from $\rho^*$ to $\rho$ along the
line $x=H(\rho^*)t$ will satisfy the Rankine-Hugoniot condition.
Due to the definition of $\rho^*$, a solution with such a jump
will also satisfy Ole{\u\i}nik's Condition, therefore it would be the unique
entropic weak solution. Notice that since $H(\lambda)<H(\rho^*)$,
no characteristics intersect along the line of discontinuity
$x=H(\rho^*)t$.  This case is called a {\sl contact
discontinuity} in \cite{ballou}. The solution is defined by
\begin{equation}\label{contact}
u(x,t)=\left\{
\begin{array}{ll}
\lambda, & x\leq H(\lambda)t;\\
h_2(x/t), & H(\lambda)t<x\leq H(\rho^*)t;\\
\rho, & H(\rho^*)t< x;\\
\end{array}
\right.
\end{equation}
Corresponding cases on the concave side of $G$ are treated
similarly.\\ \\
Let us illustrate this description on
a reference example, \textit{the totally asymmetric 2-step exclusion} 
(what follows is taken from \cite[Section 4.1.1]{bgrs1}):\\
Its flux function $G_2(u) = u+u^2 - 2 u^3$ is strictly 
convex in $0 \le u<1/6$ and strictly 
concave in $1/6< u\le 1$. For $w< 1/6$,
$w^*=(1-2w)/4$, and for $w> 1/6$, $w_*=(1-2w)/4$;
$h_{1}(x)=(1/6)(1-\sqrt{7-6x})$ for $x\in(-\infty,7/6)$, and
$h_{2}(x)=(1/6)(1+\sqrt{7-6x})$ for $x\in(7/6,+\infty)$.
 We reproduce here \cite[Figure 1]{bgrs1}, 
which  shows the six possible behaviors of the (self-similar)
solution $u(v,1)$, namely a rarefaction fan with either an increasing or a decreasing
initial condition, a decreasing shock, an increasing shock, and
a contact discontinuity with either an increasing or a decreasing initial condition. 
 Cases (a) and (b) present respectively a
rarefaction fan with increasing initial condition and a preserved
decreasing shock. These situations as well as cases (c) and (f)
cannot occur for simple exclusion. Observe also that $\rho\geq
1/2$ implies $\rho_{*}\leq 0$, which leads only to cases (d),(e),
and excludes case (f) (going back to a simple exclusion behavior).
\begin{figure}[htbp]
\hspace*{-0.8cm}%
\vspace*{-0.1cm}%
\includegraphics{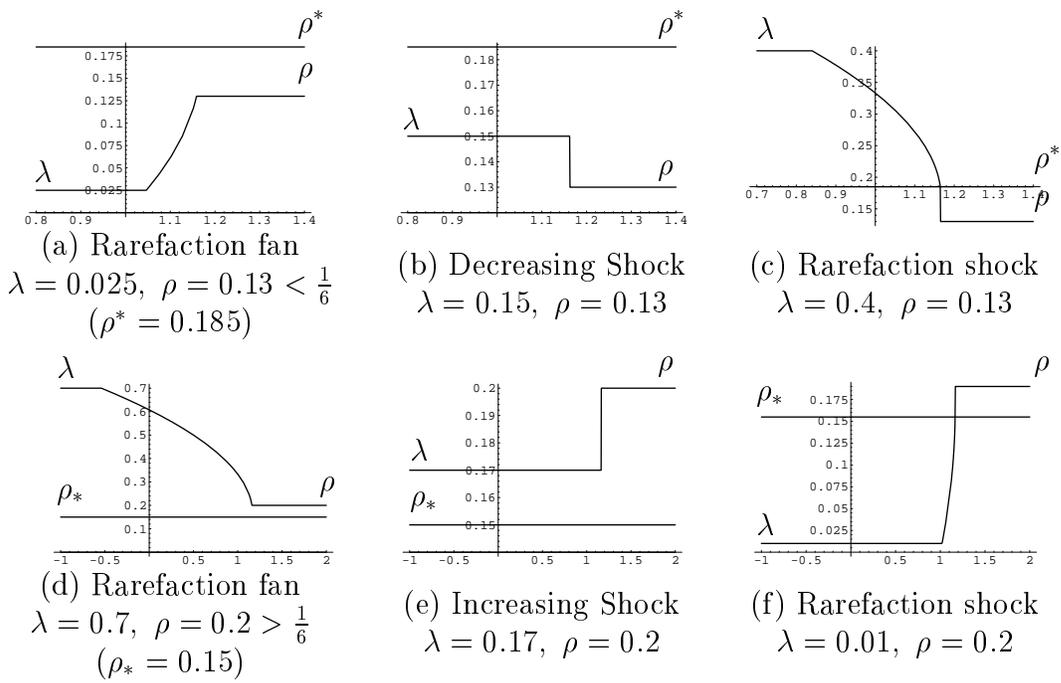}
\caption{Hydrodynamic behavior of the 2-step exclusion with
Riemann initial conditions, graph of
$u(x/t,1)$.}\label{fig:intersections}
\end{figure}  
\subsection{From Riemann to Cauchy problem}
\label{subsec:glimm}
 The beginning of this subsection is based on \cite[Section 2.4]{bgrs2}. 
We will briefly explain here the principle of approximation
schemes based on Riemann solutions, the most important of which is
probably Glimm's scheme, introduced in \cite{glimm}. Consider as
initial datum a piecewise constant profile with finitely many
jumps. The key observation is that, for small enough times, this
can be viewed as a succession of noninteracting Riemann problems.
To formalize this, we recall part of \cite[Lemma 3.4]{bgrs1}, 
which is a consequence of the finite propagation property for \eqref{hydrodyn},
see statement iv) of Proposition \ref{standard}. We denote by
$R_{\lambda,\rho}(x,t)$ the entropy solution to the Riemann
problem with initial datum \eqref{eq:rie}.
\begin{lemma}
\label{nointeraction}  (\cite[Lemma 2.1]{bgrs2}).  Let
$x_0=-\infty<x_1<\cdots<x_n<x_{n+1}=+\infty$, and
$\varepsilon:=\min_k(x_{k+1}-x_k)$. Consider the Cauchy datum
\[
u_0:=\sum_{k=0}^n r_k\mathbf 1_{(x_k,x_{k+1})}
\]
where $r_k\in[0,K]$. Then for $t<\varepsilon/(2V)$, with $V$ given
by \eqref{speed_pde}, the entropy solution
$u(.,t)$ at time $t$ coincides with $R_{r_{k-1},r_k}(.-x_k,t)$ on
$(x_{k-1}+Vt,x_{k+1}-Vt)$. In particular, $u(.,t)$ has constant
value $r_k$ on $(x_k+Vt,x_{k+1}-Vt)$.
\end{lemma}
Given some Cauchy datum $u_0$, we construct an approximate
solution $\tilde{u}(.,.)$ for the corresponding entropy solution
$u(.,.)$. To this end we define an approximation scheme based on a
time discretization step $\Delta t>0$ and a space discretization
step $\Delta x>0$. In the limit we let $\Delta x\to 0$ with the
ratio $R:=\Delta t/\Delta x$ kept constant, under the condition
\be \label{cfl} R\leq 1/(2V) \ee
known as the {\em Courant-Friedrichs-Lewy (CFL) condition}. Let
$t_k:=k\Delta t$ denote discretization times.
We start with $k=0$, setting $\tilde{u}_0^-:=u_0$.\\ \\
{\em Step one} (approximation step): Approximate $\tilde{u}_k^-$
with a piecewise constant profile $\tilde{u}_k^+$ whose step lengths  
are bounded below by $\Delta x$.\\ \\
{\em Step two} (evolution step): For $t\in[t_k,t_{k+1})$, denote
by $\tilde{u}_k(.,t)$ the entropy solution at time $t$ with
initial datum $\tilde{u}_k^+$ at time $t_k$. By \eqref{cfl} and
Lemma \ref{nointeraction}, $\tilde{u}_{k}(.,t)$ can be computed
solving only Riemann problems. Set
$\tilde{u}_{k+1}^-=\tilde{u}_k(.,t_{k+1})$. \\ \\
{\em Step three} (iteration): increment $k$ and go back to step
one.\\ \\
The approximate entropy solution is then defined by
\be \label{approximate_solution} \tilde{u}(.,t):=\sum_{k\in\N}
\tilde{u}_k(.,t){\bf 1}_{[t_k,t_{k+1})}(t)
 \ee
The efficiency of the scheme depends on how the approximation step
is performed. In Glimm's scheme, the approximation $\tilde{u}_k^+$
is defined as
\be \label{sampling} \tilde{u}_k^+:=\sum_{j\in
k/2+\Z}\tilde{u}_k^-((j+a_k/2)\Delta x){\bf 1}_{((j-1/2)\Delta
x,(j+1/2)\Delta x)} \ee
where $a_k\in(-1,1)$.
Then we have the following convergence result.
\begin{theorem}
\label{glimm}  (\cite[Theorem 2.3]{bgrs2}). 
Let $u_0$ be a given measurable initial datum. Then
every sequence $\varepsilon_n\downarrow 0$ as $n\to\infty$ has a
subsequence $\delta_n\downarrow 0$ such that, for a.e. sequence
$(a_k)$ w.r.t. product uniform measure on $(-1,1)^{\Z^+}$, the
Glimm approximation defined by \eqref{approximate_solution} and
\eqref{sampling} converges to $u$ in $L^1_{\rm
loc}(\R\times\R^{+*})$ as $\Delta x=\delta_n\downarrow 0$.
\end{theorem}
When $u_0$ has locally bounded variation, the above result is a
specialization to scalar conservation laws of a more general
result for systems of conservation laws: see Theorems 5.2.1,
5.2.2, 5.4.1 and comments following Theorem 5.2.2 in \cite{serre}.
In \cite[Appendix B]{bgrs2}, we prove that it is enough to assume $u_0$ measurable.\\ \\
Due to the nature of the approximation step \eqref{sampling}, the
proof of Theorem \ref{glimm} does not proceed by direct estimation
of the error between $\tilde{u}_k^\pm$ and $u(.,t_k)$, but
indirectly, by showing that limits of the scheme satisfy
\eqref{dissipated_entropy}.\\ \\
We will now present a different Riemann-based approximation procedure,
 introduced first in \cite[Lemma 3.6]{bgrs1}, and refined in \cite{bgrs2,bgrs3}. 
 This approximation allows direct control of the error by
using the distance $\Delta$ defined in \eqref{def_delta}. Intuitively, errors accumulate during
approximation steps, but might be amplified by the resolution
steps. The key properties of our approximation are that
the total error accumulated during the approximation step is
negligible as $\varepsilon\to 0$, and the error is not
amplified by the resolution step, because $\Delta$ does not increase along 
entropy solutions, see Proposition \ref{standard}\textit{(v)}. 
\begin{theorem}(\cite[Theorem 3.1]{bgrs1}).
\label{th_glimm}
Assume  $(T_t)_{t\geq 0}$ is a semigroup on the set of bounded 
$\mathcal R$-valued functions, with the following properties:\\ \\
(1) For any Riemann initial condition $u_0$, $t\mapsto u_t=T_t u_0$ is 
the entropy solution to \eqref{hydrodyn}
with Cauchy datum $u_0$.\\ \\
(2) (Finite speed of propagation). There is a constant v such that, for any $a,b\in\R$, any two
initial conditions $u_0$ and $u_1$ coinciding on $[x; y]$, and any $t<(b-a)/(2v)$,
$u_t=T_t u_0$ and $v_t=T_t v_0$ coincide on [x + vt; y − vt].\\ \\
(3) (Time continuity). For every bounded initial condition $u_0$ with bounded support and every $t\geq 0$,
$lim_{\varepsilon\to 0+}\Delta(T_t u_0,T_{t+\varepsilon}u_0)=0$.\\ \\
(4) (Stability). For any bounded initial conditions $u_0$ and $v_0$, with bounded support, 
$\Delta(T_t u_0,T_t v_0)\leq\Delta(u_0,v_0)$.\\ \\
Then, for any bounded $u_0$, $t\mapsto T_t u_0$ is the entropy solution to \eqref{hydrodyn} with Cauchy
datum $u_0$.
\end{theorem}
A crucial point is that properties (1)--(4) in the above Theorem  \ref{th_glimm}  
hold at particle level. 
This will allow us to mimic the scheme at particle level. The proof of Theorem  \ref{th_glimm}   
relies on the following uniform approximation
(in the sense of distance $\Delta$) by step functions, 
which is also important at particle level.
\begin{lemma}  (\cite[Lemma 4.2]{bgrs3}). 
\label{cor_approx} Assume $u_0(.)$ is a.e. $\mathcal R$-valued, has
bounded support and finite variation, and let $(x,t)\mapsto u(x,t)$
be the entropy solution to \eqref{hydrodyn} with Cauchy datum
$u_0(.)$. For every $\eps>0$,  let ${\mathcal P}_\eps$ be the set of
piecewise constant $\mathcal R$-valued functions on $\R$ with
compact support and step lengths at least $\eps$, and  set
\[
\delta_\eps(t):=\eps^{-1}\inf\{ \Delta(u(.),u(.,t)):\,
u(.)\in{\mathcal P}_\eps \}
\]
Then there is a sequence $\eps_n\downarrow 0$ as $n\to\infty$ such
that $\delta_{\eps_n}$ converges to $0$ uniformly on any bounded
subset of $\R^+$.
\end{lemma}
\section{ Proof of hydrodynamics}
\label{proof_hydro}
In this section,   based on \cite[Section 4]{bgrs4}, 
we prove the hydrodynamic limit in the quenched disordered setting,
that is Theorem \ref{th:hydro},  following the
strategy   
introduced in \cite{bgrs1,bgrs2} and significantly strengthened 
in \cite{bgrs3} and \cite{bgrs4}. First, we prove the hydrodynamic 
limit for  ${\mathcal R}^Q$-valued  Riemann initial conditions (the so-called Riemann
problem), and then use a constructive scheme to mimic the proof of 
Theorem \ref{th_glimm} at microscopic level.
\subsection{Riemann problem}\label{riemann}
Let $\lambda,\rho\in{\mathcal R}^Q$ with $\lambda<\rho$ (for $\lambda>\rho$ replace
infimum with supremum below).  We first need to derive hydrodynamics for the Riemann initial condition 
$R_{\lambda,\rho}$ defined in \eqref{eq:rie}. 
Microscopic    Riemann   states with profile \eqref{eq:rie} can be constructed
using the following lemma.
\begin{lemma}\label{big_coupling}  (\cite[Lemma 4.1]{bgrs4}). 
There exist random variables $\alpha$ and
 $(\eta^\rho:\,\rho\in{\mathcal R}^Q)$  on a probability space
$(\Omega_{\mathbf A},\mathcal F_{\mathbf A},\Prob_{\mathbf A})$ such
that
\beq\label{marginals}(\alpha,\eta^\rho)\sim \nu^{Q,\rho}, && \alpha\sim
Q\\ \label{strassen} \Prob_{\mathbf A}-a.s.,&&
\rho\mapsto\eta^\rho\mbox{ is nondecreasing}\eeq
\end{lemma}
Let  $\overline{\nu}^{Q,\lambda,\rho}$  denote the distribution of
$(\alpha,\eta^\lambda,\eta^\rho)$, and
$\overline{\nu}^{\lambda,\rho}_\alpha$ the conditional distribution of
$(\alpha,\eta^\lambda,\eta^\rho)$ given $\alpha$.
 Recall the definition \eqref{def:spacetimeshift} of the space-time shift 
 $\theta_{x_0,t_0}$ on $\Omega$ for $(x_0,t_0)\in\Z\times\R^+$. 
We now  introduce an extended shift $\theta'$ on $\Omega'={\mathbf
A}\times{\mathbf X}^2\times\Omega$. If
$\omega'=(\alpha,\eta,\xi,\omega)$
%
denotes a generic element of $\Omega'$, we set
\be\label{extended_shift}\theta'_{x,t}\omega' =
(\tau_x\alpha,\tau_x\eta_t(\alpha,\eta,\omega),\tau_x\eta_t(\alpha,\xi,\omega),\theta_{x,t}\omega)
\ee
 It is important to note that this shift incorporates disorder.
Let $T:{\mathbf X}^2\to{\mathbf X}$ be given by
\be\label{transfo_T}
T(\eta,\xi)(x)=\eta(x){\bf 1}_{\{x< 0\}}+\xi(x){\bf 1}_{\{x\geq
0\}} \ee
%
%
 A strong (that is almost sure with respect to the Poisson space) 
form of hydrodynamic limit for Riemann data can now be stated as follows.
\begin{proposition}
\label{corollary_2_2}  (\cite[Proposition 4.1]{bgrs4}).  Set, for $t\ge 0$,
\be \label{def_empirical_shift}
\beta^N_t(\omega')(dx):=\pi^N(\eta_t(\alpha,T(\eta,\xi),\omega))(dx)
\ee
For all $t>0$, $s_0\geq 0$ and $x_0\in\R$, we have that, for
$Q$-a.e. $\alpha\in{\mathbf A}$,
\[ 
\lim_{N\to\infty}\beta^N_{Nt}(\theta'_{\lfloor Nx_0\rfloor ,Ns_0}\omega')(dx)=R_{\lambda,\rho}(.,t)dx,
\quad\overline{\nu}_\alpha^{\lambda,\rho}\otimes\Prob\mbox{-a.s.}
\] 
\end{proposition}
 Proposition \ref{corollary_2_2} will follow from a law of large
numbers for currents.
Let $x_.=(x_t,\,t\geq 0)$ be a $\Z$-valued {\em cadlag} random path,
with $\abs{x_t-x_{t^-}}\leq 1$, independent of the Poisson measure
$\omega$.
We define the particle current seen by an observer travelling along
this path by
\be \label{current_3} \varphi^{x_.}_t(\alpha,\eta_0,\omega)
=\varphi^{x_.,+}_t(\alpha,\eta_0,\omega)
-\varphi^{x_.,-}_t(\alpha,\eta_0,\omega)+\widetilde{\varphi}^{x_.}_t(\alpha,\eta_0,\omega)
\ee
where
$\varphi^{x_.,\pm}_t(\alpha,\eta_0,\omega)$ count the number of
rightward/leftward crossings  of $x_.$ due to particle
jumps, and
$\widetilde{\varphi}^{x_.}_t(\alpha,\eta_0,\omega)$ is the current due
to the self-motion of the observer.
We shall  write  $\varphi^v_t$ in the particular case $x_t=\lfloor
vt\rfloor$.
Set $\phi^{v}_t(\omega'):=\varphi^{v}_t(\alpha,T(\eta,\xi),\omega)$.
Note that  for $(v,w)\in\R^2$,
\be\label{diff_currents_micro}\beta^N_{Nt}(\omega')([v,w])
=t(Nt)^{-1}(\phi^{v/t}_{Nt}(\omega')-\phi^{w/t}_{Nt}(\omega'))\ee
 We view \eqref{diff_currents_micro} as a microscopic analogue
 of \eqref{limiting_current}. Thus,  Proposition \ref{corollary_2_2} 
 boils down to showing that each term of \eqref{diff_currents_micro} 
 converges to its counterpart in \eqref{limiting_current}.
\begin{proposition}
\label{proposition_2_2}  (\cite[Proposition 4.2]{bgrs4}). 
For all $t>0$, $a\in\R^+,b\in\R$ and  $v\in\R$,
\be\lim_{N\to\infty}(Nt)^{-1}\phi^{v}_{Nt}(\theta'_{\lfloor b N\rfloor ,a N}\omega') =
\mathcal G_v(\lambda,\rho)\qquad\overline{\nu}^{Q,\lambda,\rho}\otimes\Prob-a.s. \label{as}\ee
 where $\mathcal G_v(\lambda,\rho)$ is defined by \eqref{limiting_current_0}. 
\end{proposition}
To prove Proposition \ref{proposition_2_2}, we introduce a
probability space $\Omega^+$, whose generic element is denoted by
$\omega^+$, on which is defined a Poisson process
$(N_t(\omega^+))_{t\ge 0}$ with intensity $\abs{v}$ ($v\in\R$).
Denote by ${\Prob}^+$ the associated probability. Set
\begin{eqnarray}\label{def_poisson}
x_s^N(\omega^+)&:=&({\rm sgn}(v))\left[N_{a N+s}(\omega^+)-N_{aN}(\omega^+)\right]\\
\label{mapping_tilde}\widetilde{\eta}^N_s(\alpha,\eta_0,\omega,\omega^+)
&:=&\tau_{x_s^N(\omega^+)}\eta_s(\alpha,\eta_0,\omega)\\
\label{mapping_tilde_alpha}
\widetilde{\alpha}^N_s(\alpha,\omega^+) & := & \tau_{x^N_s(\omega^+)}\alpha
\end{eqnarray}
Thus  $(\widetilde{\alpha}_s^N,\widetilde{\eta}_s^N)_{s\ge 0}$ is a
Feller process with generator
\[
 L^v= {\mathfrak L} +S^v,\quad S^v f(\alpha,\zeta)=\abs{v}
[f(\tau_{{\rm sgn}(v)}\alpha,\tau_{{\rm
sgn}(v)}\zeta)-f(\alpha,\zeta)] \]
 for $f$ local and $\alpha\in\mathbf A$, $\zeta\in{\mathbf X}$.
Since any translation invariant measure on ${\mathbf
A}\times{\mathbf X}$ is stationary for the pure shift generator
$S^v$, we have
 ${\mathcal I}_{\mathfrak L}\cap{\mathcal S}={\mathcal
I}_{L^v}\cap{\mathcal S}$. 
Define the time and space-time empirical measures (where $\eps>0$) by
\begin{eqnarray} \label{def_empirical} 
m_{tN}(\omega',\omega^+)&:=&(Nt)^{-1}\int_0^{tN}
\delta_{(\widetilde{\alpha}^N_s(\alpha,\omega^+),
\widetilde{\eta}^N_s(\alpha,T(\eta,\xi),\omega,\omega^+))}ds\\
\label{def_empirical_2}
m_{tN,\eps}(\omega',\omega^+)&:=&\abs{\Z\cap[-\eps N,\eps
N]}^{-1}\sum_{x\in\Z:\,\abs{x}\leq \eps
N}\tau_xm_{tN}(\omega',\omega^+) 
\end{eqnarray}
Notice  that there is a disorder component we cannot omit in the empirical measure, 
although ultimately  we are only interested
in the behavior of the $\eta$-component.
Let  ${\mathcal M}_{\lambda,\rho}^Q$  denote
the compact set   of probability measures 
 $\mu(d\alpha,d\eta)\in \mathcal I_{\mathfrak L}\cap\mathcal S$  such that
$\mu$ has $\alpha$-marginal $Q$, and
 $\nu^{Q,\lambda}\ll\mu\ll\nu^{Q,\rho}$.  By Proposition \ref{invariant},
\be\label{setofmeasures}
\mathcal M_{\lambda,\rho}^Q=\left\{\nu(d\alpha,d\eta)
=\int\nu^{Q,r}(d\alpha,d\eta)\gamma(dr):\,
\gamma\in\mathcal P([\lambda,\rho]\cap{\mathcal R}^Q)\right\}
\ee
The key ingredients for Proposition \ref{proposition_2_2} are the
following lemmas.
\begin{lemma}\label{current_comparison}  (\cite[Lemma 4.2]{bgrs4}). 
The function $\phi^v_t(\alpha,\eta,\xi,\omega)$ is increasing in $\eta$, decreasing in $\xi$.
\end{lemma}
\begin{lemma}\label{lemma_empirical}  (\cite[Lemma 4.3]{bgrs4}). 
With
$\overline{\nu}^{Q,\lambda,\rho}\otimes{\Prob}\otimes{\Prob}^+$-probability
one, every subsequential limit as $N\to\infty$ of
$m_{tN,\eps }(\theta'_{\lfloor b N\rfloor ,a
N}\omega',\omega^+)$ lies in $\mathcal M_{\lambda,\rho}^Q$.
\end{lemma}
 Lemma \ref{current_comparison} is a consequence of the monotonicity property \eqref{attractive_1}.
Lemma \ref{lemma_empirical} relies in addition
on a space-time ergodic theorem and on a general uniform large deviation 
upper bound for space-time empirical measures of Markov processes. 
We state these two results before proving Proposition \ref{proposition_2_2}.
\begin{proposition} (\cite[Proposition 2.3]{bgrs3}).
\label{prop_ergo} Let $(\eta_t)_{t\ge 0}$ be a Feller process on
$\mathbf X$ with  a translation invariant generator $L$,  that is
\be\label{translation_gen}\tau_1 L\tau_{-1}=L\ee
Assume further that
\[
\mu\in({\mathcal I}_L\cap{\mathcal S})_e\]
where ${\mathcal I}_L$ denotes the set of invariant measures for
$L$. Then, for any local function $f$ on $\mathbf X$, and any $a>0$
\be \label{ergo}\lim_{\ell\to\infty}\frac{1}{a\ell^2}\int_0^{a\ell}
\sum_{i=0}^\ell\tau_i f(\eta_t)dt=\int f
d\mu=\lim_{\ell\to\infty}\frac{1}{a\ell^2}\int_0^{a\ell}
\sum_{i=-\ell}^{-1}\tau_i f(\eta_t)dt\ee
a.s. with respect to the law of the process with initial
distribution $\mu$.
\end{proposition}
\begin{lemma}\label{deviation_empirical} (\cite[Lemma 3.4]{bgrs3})
Let ${\bf P}_\nu^v$ denote the law of a Markov process $(\widetilde\alpha_.,\widetilde\xi_.)$ with
generator $L^v$ and initial distribution $\nu$. For $\eps>0$, let
\be \label{def_empirical-gen}
\pi_{t,\eps}:=\abs{\Z\cap[-\eps t,\eps
t]}^{-1}\sum_{x\in\Z\cap[-\eps t,\eps t]}t^{-1}\int_0^t\delta_{(\tau_x\widetilde\alpha_s,\tau_x\widetilde\xi_s)}ds
\ee
Then, there exists a functional ${\mathcal D}_v$ which is
nonnegative, l.s.c., and satisfies
${\mathcal D}_v^{-1}(0)={\mathcal I}_{L^v}$, such that,
for every closed subset $F$ of ${\mathcal P}({\mathbf A}\times\bf
X)$,
\be \label{ld} \limsup_{t\to\infty}t^{-1}\log\sup_{\nu\in\mathcal P({\bf
A}\times{\mathbf X})}{\bf
P}_\nu^v\left(\pi_{t,\eps}(\widetilde\xi_.)\in F\right)\leq
-\inf_{\mu\in F}{\mathcal D}_v(\mu) \ee
\end{lemma}
\begin{proof}{Proposition}{proposition_2_2}
We will show that
\beq \liminf_{N\to\infty}\,(Nt)^{-1}\phi^{v}_{tN}\circ\theta'_{\lfloor b N\rfloor ,a
N}(\omega') & \geq & \mathcal G_v(\lambda,\rho),\quad
\overline{\nu}^{Q,\lambda,\rho}\otimes\Prob\mbox{-a.s.}\label{liminf_better}\\
\label{limsup}
\limsup_{N\to\infty}\,(Nt)^{-1}\phi^{v}_{tN}\circ\theta'_{\lfloor b N\rfloor ,a
N}(\omega') & \leq & \mathcal G_v(\lambda,\rho),\quad \overline{\nu}^{Q,\lambda,\rho}\otimes\Prob\mbox{-a.s.}
\eeq
{\em Step one: proof of \eqref{liminf_better}.}\par\noindent
  Setting
$\varpi_{a N}=\varpi_{a
N}(\omega'):=T\left( \tau_{\lfloor b N\rfloor }\eta_{a
N}(\alpha,\eta,\omega), \tau_{\lfloor b N\rfloor }\eta_{a N}(\alpha,\xi,\omega)
 \right)$,
we have
\be \label{sothat} (Nt)^{-1}\phi^v_{tN}\circ\theta'_{\lfloor b N\rfloor ,a
N}(\omega')=(Nt)^{-1}\varphi^{v}_{tN}(\tau_{\lfloor b
N\rfloor }\alpha,\varpi_{a N},\theta_{\lfloor b
N\rfloor ,a N}\omega) \ee
Let,   for every $(\alpha,\zeta,\omega,\omega^+)\in{\mathbf
A}\times{\mathbf X}\times\Omega\times\Omega^+$ and $x^N_.(\omega^+)$
given by \eqref{def_poisson},
\be\label{psi-t-v}
\psi_{tN}^{v,\eps}(\alpha,\zeta,\omega,\omega^+):=\abs{\Z\cap[-\eps
N,\eps N]}^{-1}\sum_{y\in\Z:\,\abs{y}\leq\eps
N}\varphi^{x^N_.(\omega^+)+y}_{tN}(\alpha,\zeta,\omega) \ee
Note that  $\lim_{N\to\infty}(Nt)^{-1}x_{tN}^N(\omega^+)= v$,  $\Prob^+$-a.s., and that for
two paths $y_.,z_.$  (see \eqref{current_3}),
\[
\abs{\varphi^{y_.}_{tN}(\alpha,\eta_0,\omega)-\varphi^{z_.}_{tN}(\alpha,\eta_0,\omega)} \leq
K\left(\abs{y_{tN}-z_{tN}}+ \abs{y_0-z_0}\right)
\]
Hence the proof of \eqref{liminf_better} reduces to that of the same
inequality where we replace $(Nt)^{-1}\phi^{v}_{tN}\circ\theta'_{\lfloor b
N\rfloor ,a N}(\omega')$ by
$ (Nt)^{-1}\psi^{v,\eps}_{tN}(\tau_{\lfloor b N\rfloor }\alpha,\varpi_{a
N},\theta_{\lfloor b N\rfloor ,a N}\omega,\omega^+)$ 
and $\overline{\nu}^{Q,\lambda,\rho}\otimes\Prob$ by
$\overline{\nu}^{Q,\lambda,\rho}\otimes\Prob\otimes\Prob^+$.
By definitions \eqref{def_f}, \eqref{current_3} of flux and current,
for any $\alpha\in\mathbf{A}$, $\zeta\in\mathbf{X}$,
\begin{eqnarray*} &&M^{x,v}_{tN}(\alpha,\zeta,\omega,\omega^+):=
\varphi^{x^N_.(\omega^+)+x}_{tN}(\alpha,\zeta,\omega)-\nonumber\\
&& \int_0^{tN} \tau_x\left\{
j(\widetilde{\alpha}^N_{s}(\alpha,\omega^+),\widetilde{\eta}^N_{s}(\alpha,\zeta,\omega,\omega^+))
-v (\widetilde{\eta}^N_{s}(\alpha,\zeta,\omega,\omega^+))({\bf
1}_{\{v>0\}})\right\}ds\label{martingale}
\end{eqnarray*}
is a mean $0$ martingale under $\Prob\otimes\Prob^+$.  Let
\beq\nonumber R_{tN}^{\eps,v}&:=&\label{as_2} \left(Nt\abs{\Z\cap[-\eps
N,\eps N]}\right)^{-1}\sum_{x\in\Z:\,\abs{x}\leq\eps N}M^{x,v}_{tN}(
\tau_{\lfloor b N\rfloor }\alpha,\varpi_{a N},\theta_{\lfloor b
N\rfloor ,a N}\omega,\omega^+)\\\nonumber
&=&(Nt)^{-1}\psi^{v,\eps}_{tN}(\tau_{\lfloor b N\rfloor }\alpha,\varpi_{a
N},\theta_{\lfloor b N\rfloor ,a N}\omega,\omega^+)\\
&&-\int [j(\alpha,\eta)-v\eta({\bf
1}_{\{v>0\}})]m_{tN,\eps}(\theta'_{\lfloor b N\rfloor ,a
N}\omega',\omega^+)(d\alpha,d\eta)
\label{replace_oncemore} \eeq
where the last equality comes from \eqref{def_empirical_2},
\eqref{psi-t-v}.
 The exponential martingale associated with $M^{x,v}_{tN}$ 
yields a Poissonian bound, uniform in $(\alpha,\zeta)$, for the
exponential moment of $M_{tN}^{x,v}$ with respect to
$\Prob\otimes\Prob^+$. Since $\varpi_{a N}$ is independent of
$(\theta_{\lfloor b N\rfloor ,a N}\omega,\omega^+)$ under
$\overline{\nu}^{Q,\lambda,\rho}\otimes\Prob\otimes\Prob^+$, the bound
is also valid under this measure, and Borel-Cantelli's lemma implies
$\lim_{N\to\infty}R_{tN}^{\eps,v}=0$. 
{}From \eqref{replace_oncemore}, Lemma \ref{lemma_empirical} and
Corollary \ref{corollary_invariant}, \textit{(ii)} imply
\eqref{liminf_better}, as well as
 \be \limsup_{N\to\infty}\,(Nt)^{-1}\phi^{v}_{tN}\circ\theta'_{\lfloor b N\rfloor ,a
N}(\omega')  \leq  \sup_{r\in[\lambda,\rho]\cap{\mathcal R}^Q}
[G^Q(r)-v r],\quad
\overline{\nu}^{Q,\lambda,\rho}\otimes\Prob\mbox{-a.s.}\label{liminf_better_2}\ee 
{\em Step two: proof of \eqref{limsup}.}
Let  $r\in[\lambda,\rho]\cap{\mathcal R}^Q$.  We define
$\overline{\nu}^{Q,\lambda,r,\rho}$ as the distribution of
$(\alpha,\eta^\lambda,\eta^r,\eta^\rho)$.
By \eqref{liminf_better} and \eqref{liminf_better_2}, 
\[
\lim_{N\to\infty}\,(Nt)^{-1}\phi^{v}_{tN}\circ\theta'_{\lfloor b N\rfloor ,a
N}(\alpha,\eta^r,\eta^r,\omega)  =   G^Q(r)-vr \]
By Lemma \ref{current_comparison},
\[
\phi^v_{tN}\circ\theta'_{\lfloor bN\rfloor ,aN}(\omega')
\leq\phi^v_{tN}\circ\theta'_{\lfloor bN\rfloor ,aN}(\alpha,\eta^r,\eta^r,\omega)\]
 The result follows by continuity of  $G^Q$  and minimizing over $r$.
\end{proof}
\subsection{Cauchy problem}
\label{Cauchy}
 Using \eqref{deltastab} and the fact that an arbitrary 
function can be approximated by a  ${\mathcal R}^Q$-valued  function 
with respect to the distance $\Delta$ defined by \eqref{def_delta}, 
the proof of Theorem \ref{th:hydro} for general initial data $u_0$ 
can be reduced (see \cite{bgrs2}) to the case of  ${\mathcal R}^Q$-valued  initial data
by coupling and approximation arguments (see  \cite[Section 4.2.2]{bgrs3}).
\begin{proposition}\label{hydro_finite}  (\cite[Proposition 4.3]{bgrs4}). 
Assume $(\eta^N_0)$ is a sequence of configurations such that:
(i) there exists $C>0$  such that  for all $N\in\N$, $\eta^N_0$ is
supported on $\Z\cap[-CN,CN]$; \\
(ii) $\pi^N(\eta^N_0)\to u_0(.)dx$ as $N\to\infty$, where $u_0$
 has compact support,  is a.e.  ${\mathcal R}^Q$-valued  and
has finite space variation. \\
Let $u(.,t)$ denote the unique entropy solution to \eqref{hydrodynamics} with Cauchy datum $u_0(.)$.
Then,  $Q\otimes\Prob$-a.s. as $N\to\infty$,
\[
\Delta^N(t):=\Delta(\pi^N(\eta^N_{Nt}(\alpha,\eta^N_0,\omega)),u(.,t)dx)
\]
converges uniformly to $0$ on $[0,T]$ for every $T>0$.
\end{proposition}
 Before proving this proposition, we state two crucial tools 
in their most complete form (see \cite{mrs}),
the macroscopic stability and the finite propagation property for the particle system. 
Macroscopic stability yields that  the distance 
$\Delta$ defined in \eqref{def_delta}  is an ``almost'' nonincreasing functional for two coupled
particle systems.  It is thus a microscopic analogue of property \eqref{deltastab} 
in Proposition \ref{standard}.  The finite propagation property is a microscopic
analogue of Proposition \ref{standard}, \textit{iv)}.  For the misanthrope's process, it follows essentially
from the finite mean assumption (M4) on page \pageref{assumptions_M}. 
\begin{proposition}
\label{prop:macro-stab} (\cite[Proposition 4.2]{bgrs3} with \cite[Theorem 2]{mrs})
Assume $p(.)$ has a finite first moment and a positive mean. 
Then there exist constants $C>0$ and
$c>0$, depending only on $b(.,.)$ and $p(.)$, such that the
following holds. For every  $N\in\N$, $(\eta_0,\xi_0)\in{\mathbf X}^2$
with
$\abs{\eta_0}+\abs{\xi_0}:=\sum_{x\in\Z}[\eta_0(x)+\xi_0(x)]<+\infty$,
and every $\gamma>0$, the event
\be\label{bmevent} \forall t>0:
\,\Delta(\pi^N(\eta_t(\alpha,\eta_0,\omega)),\pi^N(\eta_t(\alpha,\xi_0,\omega)))
\leq \Delta(\pi^N(\alpha,\eta_0),\pi^N(\alpha,\xi_0))+\gamma
\ee
has $\Prob$-probability at least
$1-C(\abs{\eta_0}+\abs{\xi_0})e^{-cN\gamma}$.
\end{proposition}
\begin{lemma}
\label{finite_prop} (\cite[Lemma 15]{mrs})
Assume $p(.)$ has a finite third moment. 
There exist constant $v$,
and function $A(.)$ (satisfying $\sum_n A(n) < \ \infty $), depending only on $b(.,.)$ and
$p(.)$, such that the following holds. For any $x,y\in\Z$, any
$(\eta_0,\xi_0)\in{\mathbf X}^2$, and any $0<t<(y-x)/(2v)$: if $\eta_0$
and $\xi_0$ coincide on the site interval $[x,y]$, then  with
$\Prob$-probability at least $1-A(t)$,
$\eta_s(\alpha,\eta_0,\omega)$ and $\eta_s(\alpha,\xi_0,\omega)$ coincide on the
site interval $[x+vt,y-vt]\cap\Z$ for every $s\in[0,t]$.
\end{lemma}
 To prove Proposition \ref{prop:macro-stab}, we work with a coupled process,
 and reduce the problem to analysing   the evolution of labeled positive and negative discrepancies 
to control their coalescences. For this we order the discrepancies, we
possibly relabel them according to their movements to favor coalescences, 
and define \textit{windows}, which are space intervals on which coalescences
are favored, in the same spirit as in \cite{bm}.  
The irreducibility assumption (that is \eqref{irreducibility_misanthrope} 
in the case of the misanthrope's process) 
plays an essential role there.  To take advantage of \cite{bm},
we treat separately the movements
of discrepancies corresponding to big jumps,  which can be controlled 
thanks to the finiteness of the first moment 
(that is, assumption \textit{(M4)} on page \pageref{assumptions_M}).    \\ 
\begin{proof}{proposition}{hydro_finite}
By initial assumption \eqref{initial_profile_vague},
$\lim_{N\to\infty}\Delta^N(0)=0$.
Let $\eps>0$, and  $\eps'=\eps/(2V)$, for  $V$  given by
\eqref{speed_pde}. Set $t_k=k\eps'$ for $k\le \kappa:=\lfloor
T/\eps'\rfloor $, $t_{\kappa+1}=T$.  Since the number of steps is
proportional to $\eps$, if we want to bound the total error, the
main step is to prove
\be\label{wearegoing_1}
\limsup_{N\to\infty}\sup_{k=0,\ldots,{\mathcal
K}-1}\left[\Delta^N(t_{k+1})-\Delta^N(t_k)\right]\leq
3\delta\eps,\quad Q\otimes\Prob\mbox{-a.s.} \ee
where  $\delta:=\delta(\eps)$  goes to 0 as $\eps$ goes to 0;
the gaps between discrete times are filled by an  
  estimate for the time
modulus of continuity of $\Delta^N(t)$ (see  \cite[Lemma 4.5]{bgrs3}).  \\ \\
{\em Proof of \eqref{wearegoing_1}}. Since $u(.,t_k)$ has locally
finite variation, by \cite[Lemma 4.2]{bgrs3}, for all $\eps>0$
 we can find  
 functions
\be \label{decomp_approx} v_k=\sum_{l=0}^{l_k}r_{k,l}{\bf
1}_{[x_{k,l},x_{k,l+1})} \ee
with
$ -\infty=x_{k,0}<x_{k,1}<\ldots<x_{k,l_k}<x_{k,l_k+1}=+\infty$,
$r_{k,l}\in{{\mathcal R}^Q}$, $r_{k,0}=r_{k,l_k}=0$,
such that $x_{k,l}-x_{k,l-1}\geq \eps$, and 
\be\label{uniform_approx}
\Delta(u(.,t_k)dx,v_kdx)  \leq  \delta\eps
\ee
 For $t_k\leq t< t_{k+1}$, we denote by $v_k(.,t)$ the entropy
solution to \eqref{hydrodynamics} at time $t$ with Cauchy datum
$v_k(.)$.
 The configuration $\xi^{N,k}$
defined on $(\Omega_{\mathbf A}\otimes\Omega,\mathcal F_{\mathbf
A}\otimes\mathcal F,\Prob_{\mathbf A}\otimes\Prob)$ (see Lemma
\ref{big_coupling}) by
\[
\xi^{N,k}(\omega_{\mathbf
A},\omega)(x):=\eta_{Nt_k}(\alpha(\omega_{\mathbf
A}),\eta^{r_{k,l}}(\omega_{\mathbf A}), \omega)(x),\, \mbox{ if }\,
\lfloor Nx_{k,l}\rfloor \leq x<\lfloor Nx_{k,l+1}\rfloor  \]
is a microscopic version of $v_k(.)$, since by Proposition
\ref{corollary_2_2} with $\lambda=\rho=r^{k,l}$,
\be \label{profile_xi}
\lim_{N\to\infty}\pi^N(\xi^{N,k}(\omega_{\mathbf A},\omega))(dx)
=v_k(.)dx,\quad\Prob_{\mathbf A}\otimes\Prob\mbox{-a.s.} \ee
 We denote by
$\xi^{N,k}_t(\omega_{\mathbf A},\omega)  =
\eta_{t}(\alpha(\omega_{\mathbf A}),\xi^{N,k}(\omega_{\mathbf
A},\omega), \theta_{0,Nt_k}\omega)$
evolution starting  from $\xi^{N,k}$.  
By triangle inequality,
\beq
\Delta^N(t_{k+1})-\Delta^N(t_k) & \leq & \Delta\left[ \pi^N(
\eta^N_{Nt_{k+1}} ),\pi^N( \xi^{N,k}_{N\eps'} )
\right]-\Delta^N(t_k)\label{decomp_delta_1}\\
& + & \Delta\left[ \pi^N( \xi^{N,k}_{N\eps'}
),v_k(.,\eps')dx \right]\label{decomp_delta_2}\\
& + & \Delta(v_k(.,\eps')dx,u(.,t_{k+1})dx)\label{decomp_delta_3}
\eeq
 To conclude, we rely on Properties 
  \eqref{equiv_delta}, \eqref{bmevent} and \eqref{deltastab}  of
$\Delta$:  Since  $\eps'=\eps/(2V)$,  finite propagation property
for \eqref{hydrodynamics} and for the particle system  (see
 Proposition \ref{standard}, \textit{iv)} and Lemma \ref{finite_prop})  and
Proposition \ref{corollary_2_2} imply
\[
\lim_{N\to\infty}\pi^N(
\xi^{N,k}_{N\eps'}(\omega_{\mathbf A},\omega)
)=v_k(.,\eps')dx,\qquad \Prob_A\otimes\Prob\mbox{-a.s.}
\]
Hence, the term \eqref{decomp_delta_2} converges a.s. to $0$ as
$N\to\infty$. By $\Delta$-stability for \eqref{hydrodynamics}, the
term \eqref{decomp_delta_3}  is bounded by 
$\Delta(v_k(.)dx,u(.,t_k)dx)\leq\delta\eps$.
We now consider the term \eqref{decomp_delta_1}.
By macroscopic stability  (Proposition \ref{prop:macro-stab}), 
outside probability $e^{-CN\delta\eps}$,
\be\label{macrostablast}\Delta\left[ \pi^N(
\eta^N_{Nt_{k+1}}),\pi^N( \xi^{N,k}_{N\eps'} ) \right]\leq
\Delta\left[ \pi^N( \eta^N_{Nt_{k}} ),\pi^N( \xi^{N,k} )
\right]+\delta\eps
\ee
Thus the event \eqref{macrostablast} holds a.s. for $N$ large
enough.
By triangle inequality,
\begin{eqnarray*}
&&\Delta\left[ \pi^N( \eta^N_{Nt_{k}} ),\pi^N( \xi^{N,k} ) \right] -
\Delta^N(t_k) \\
&\leq&\Delta\left(u(.,t_k)dx,v_k(.)dx\right)+\Delta\left[v_k(.)dx,\pi^N(\xi^{N,k})\right]
\end{eqnarray*}
for which \eqref{uniform_approx}, \eqref{profile_xi} yield  as
$N\to\infty$ an upper bound $2\delta\varepsilon$, hence
$3\delta\varepsilon$ for the term \eqref{decomp_delta_1}.
\end{proof}
\section{Other models  under  a general framework}\label{sec:othermodels}
 As announced in Section \ref{sec:results}, 
we first define in Subsection \ref{subsec:framework}, 
as in  \cite{ba, bgrs4,swart}  
a general framework (which encompasses all known examples), that we
illustrate with  our reference  examples, 
the misanthropes process  (with generator \eqref{generator}), 
and the $k$-exclusion process (with generator \eqref{generator_k}). 
Next in Subsection \ref{subsec:examples}, we study examples 
of more complex models thanks to this new framework. 
\subsection{Framework}\label{subsec:framework}
This section is based on parts of \cite[Sections 2, 5]{bgrs4}.
The interest of an abstract description is to summarize all details of the 
microscopic dynamics in a single 
mapping, hereafter denoted by $\mathcal T$.  This mapping contains 
both the generator description of the dynamics and its graphical construction.
Once a given model is written in this framework, 
all proofs can be done without any model-specific computations, 
only relying on the properties of $\mathcal T$. \\ \\
\textbf{Monotone transformations.} 
Given an environment $\alpha\in{\mathbf A}$, we are going to define a Markov generator 
whose associated dynamics can be generically understood as follows: we pick a location 
on the lattice and around this location, apply a random monotone transformation to 
the current configuration. Let $(\mathcal V, {\mathcal F}_{\mathcal V}, m)$ 
be a measure space, where $m$ is a nonnegative finite measure. This space will be 
used to generate a monotone conservative transformation, that is a mapping 
$\mathcal T:{\mathbf X}\to{\mathbf X}$ such that:\\
\textit{(i)} $\mathcal T$ is nondecreasing: that is, for every $\eta\in{\mathbf X}$ 
and $\xi\in\mathbf X$, $\eta\leq\xi$ implies $\mathcal T\eta\leq\mathcal T\xi$;\\
\textit{(ii)} $\mathcal T$ acts on finitely many sites, that is, there exists a finite 
subset $S$ of $\Z$ such that, for all $\eta\in\mathbf X$, $\mathcal T\eta$ 
only depends on the restriction of $\eta$ to $S$, and coincides with $\eta$ outside $S$;\\
 \textit{(iii)}  $\mathcal T\eta$ is conservative,  that is,  for every $\eta\in{\mathbf X}$,
\be\label{conservative}
\sum_{x\in S}{\mathcal T}\eta(x)=\sum_{x\in S}\eta(x)
\ee
We denote by $\mathfrak T$ the set of monotone conservative transformations,
 endowed with the $\sigma$-field generated by the evaluation mappings 
$\mathcal T\mapsto\mathcal T\eta$ for all $\eta\in\mathbf X$.  \\ 
\noindent
\textbf{Definition of the dynamics.} 
In order to define the process, we specify a mapping 
\[
{\mathbf A}\times{\mathcal V}\to{\mathfrak T},\quad
(\alpha,v)\mapsto\mathcal T^{\alpha,v}
\]
such that  
for every $\alpha\in\mathbf A$ and $\eta\in{\mathbf X}$, 
the mapping $v\mapsto \mathcal T^{\alpha,v}\eta$ is measurable from 
$(\mathcal V,{\mathcal F}_{\mathcal V}, m)$  to $(\mathbf X, \mathcal F_0)$. 
When $m$ is a probability measure, this amounts to saying that 
for each $\alpha\in\mathbf A$, the mapping $v\mapsto \mathcal T^{\alpha,v}$ 
is a $\mathfrak T$-valued random variable.
The transformation $\mathcal T^{\alpha,v}$ must be understood 
as applying a certain update rule around $0$ to the current configuration, 
depending on the environment around $0$.  If  $x\in\Z\setminus\{0\}$,  we define
\be\label{shifted_transfo}
\mathcal T^{\alpha,x,v}:=\tau_{-x}T^{\tau_x\alpha,v}\tau_x
\ee
This definition can be understood as applying the same update rule around 
site  $x$,  which involves simultaneous shifts of the initial 
environment and transformation.\\ \\
We now define the Markov generator
\be\label{gengen}
L_\alpha f(\eta)=\sum_{x\in\Z}\int_{\mathcal V}\left[
f\left({\mathcal T}^{\alpha,x,v}\eta\right)-f(\eta)
\right]m(dv) \ee
As a result of \eqref{shifted_transfo}, the generator \eqref{gengen} 
satisfies the commutation property  \eqref{commutation}.\\ \\ 
 \textbf{Basic examples.}  To illustrate the above framework, 
 we come back to our reference examples of
Section \ref{sec:story}.
\\ \\
\textit{The Misanthrope's process.}
Let
\be\label{special_choice}
\mathcal V:=\Z\times[0,1],\quad
v=(z,u)\in\mathcal V,\quad
m(dv)=c^{-1}||b||_\infty p(dz)\lambda_{[0,1]}(du)
\ee
For $v=(z,u)\in\mathcal V$,  ${\mathcal T}^{\alpha,v}$ is defined by 
\be\label{update_misanthrope}
{\mathcal T}^{\alpha,v}\eta=\left\{
\ba{lll}
\eta^{0,z} & \mbox{if} & u<\displaystyle\alpha(0){\frac{b(\eta(0),\eta(z))}
{c^{-1}||b||_\infty}}\\
\eta &  & \mbox{otherwise}
\ea
\right.
\ee
Once given ${\mathcal T}^{\alpha,v}$ in \eqref{update_misanthrope},
we deduce $\mathcal T^{\alpha,x,v}$ from \eqref{shifted_transfo}:
\be\label{update_misanthrope_x}
{\mathcal T}^{\alpha,x,v}\eta=\left\{
\ba{lll}
\eta^{x,x+z} & \mbox{if} & u<\displaystyle\alpha(x){\frac{b(\eta(x),\eta(x+z))}
{c^{-1}||b||_\infty}}\\
\eta &  & \mbox{otherwise}
\ea
\right.
\ee
Though  $\mathcal T^{\alpha,v}$ is the actual input of the model, 
from which ${\mathcal T}^{\alpha,x,v}$ follows, in the forthcoming examples, 
for the sake of readability, we will directly define ${\mathcal T}^{\alpha,x,v}$.\\ \\
Monotonicity of the transformation $\mathcal T^{\alpha,x,z}$ given in 
\eqref{update_misanthrope_x} follows from assumption {\em (M2)}  on page \pageref{assumptions_M}. 
One can deduce from \eqref{gengen} and \eqref{update_misanthrope_x} 
 that $L_\alpha$ is indeed given by \eqref{generator}. 
\\ \\
 {\em The $k$-step exclusion process.} Here we let $\mathcal V=\Z^k$ 
and $m$ denote the distribution of the first $k$ steps of a random walk 
on $\Z$ with increment distribution $p(.)$ absorbed at $0$. We define, 
for $(x,v,\eta)\in\Z\times{\mathcal V}\times{\mathbf X}$, with $v=(z_1,\ldots,z_k)$,
\[
N(x,v,\eta)=\inf\{
i\in\{1,\ldots,k\}:\,\eta(x+z_i)=0
\}
\]
with the convention that $\inf\emptyset=+\infty$. We then set
\be\label{transfo_kstep}
{\mathcal T}^{\alpha,x,v}\eta=\left\{
\ba{lll}
\eta^{x,x+N(x,v,\eta)} & \mbox{if} & N(x,v,\eta)<+\infty\\
\eta  & \mbox{if} & N(x,v,\eta)=+\infty
\ea
\right.
\ee
One can show that this transformation is monotone,  either directly,
or by application of Lemma \ref{attractive_kstep}, since the $k$-step exclusion 
is a particular $k$-step misanthropes process (see special case 2a below). 
Plugging \eqref{transfo_kstep} into \eqref{gengen} yields \eqref{generator_k}. \\ \\ 
We now describe the so-called \textit{graphical construction} of the system given by
\eqref{gengen}, that is its pathwise construction on a Poisson space.
We consider the probability space $(\Omega,{\mathcal F},\Prob)$ of
locally finite point measures $\omega(dt,dx,dv)$ on
$\R^+\times\Z\times\mathcal V$, where $\mathcal F$ is generated by the
mappings $\omega\mapsto\omega(S)$ for Borel sets $S$  of 
$\R^+\times\Z\times\mathcal V$,  and $\Prob$
makes $\omega$ a Poisson process with intensity
\[
M(dt,dx,dv)=
\lambda_{\R^+}(dt)\lambda_\Z(dx)m(dv) \]
  denoting by $\lambda$  either   the Lebesgue or
 the counting measure. We write $\Exp$  for  expectation 
 with respect to $\Prob$. 
There exists a unique mapping 
\be \label{unique_mapping} (\alpha,\eta_0,t)\in{{\mathbf
A}}\times{\mathbf
X}\times\R^+\mapsto\eta_t=\eta_t(\alpha,\eta_0,\omega)\in{\mathbf X}
\ee
satisfying: \textit{(a)} $t\mapsto\eta_t(\alpha,\eta_0,\omega)$ is
right-continuous; \textit{(b)} $\eta_0(\alpha,\eta_0,\omega)=\eta_0$; \textit{(c)} the
particle configuration  is updated at points  $(t,x,v)\in\omega$ (and only at such  points;
by $(t,x,v)\in\omega$ we mean  $\omega\{(t,x,v)\}=1$)  according to the rule
\be\label{update_rule}
\eta_t(\alpha,\eta_0,\omega)={\mathcal T}^{\alpha,x,v}\eta_{t^-}(\alpha,\eta_0,\omega)
\ee
The processes defined by \eqref{gengen} and \eqref{unique_mapping}--\eqref{update_rule} 
exist and are equal in law under general conditions given in \cite{swart}, 
see also \cite{bmrs2} for a summary of this construction). \\ \\
\textbf{Coupling and monotonicity.}
The monotonocity of ${\mathcal T}^{\alpha,x,u}$ implies monotone dependence
\eqref{attractive_1} with respect to the initial state.  Thus, an arbitrary number
of processes can be coupled via the graphical construction. This implies complete monotonicity and thus attractiveness. It is also possible to define the coupling
of any number of processes using the tranformation $\mathcal T$.
For instance, in order to couple two processes, we define 
the coupled generator  $\overline{L}_\alpha$ on ${\mathbf X}^2$ by
\be\label{coupling_t}
\overline{L}_\alpha f(\eta,\xi):=\sum_{x\in\Z}\int_{\mathcal V}\left[
f\left({\mathcal T}^{\alpha,x,v}\eta,{\mathcal T}^{\alpha,x,v}\xi\right)-f(\eta,\xi)
\right]m(dv)
\ee
for any local function $f$ on ${\mathbf X}^2$. 
\subsection{ Examples }\label{subsec:examples}
 We refer the reader to \cite[Section 5]{bgrs4} for various examples 
of completely monotone models defined using this framework. 
We now review two of the models introduced in \cite[Section 5]{bgrs4}, 
then present  a new model containing all the other models in this paper,
 \textit{the $k$-step misanthropes process}.  \\ \\
\textbf{The generalized misanthropes' process.}  (\cite[Section 5.1]{bgrs4}). 
Let $K\in\N$. Let $c\in(0,1)$, and  $p(.)$ (resp. $P(.)$), be a probability
distribution on $\Z$.  Define $\mathbf A$ to be the set of  functions
 $B:\Z^2\times\{0,\ldots,K\}^2\to\R^+$ such that:\\ \\
{\em (GM1)} For all $(x,z)\in\Z^2$,  $B(x,z,.,.)$ satisfies 
assumptions \textit{(M1)--(M3)}  on page \pageref{assumptions_M};\\  
{\em (GM2)} There exists a constant $C>0$ and a probability measure $P(.)$ on $\Z$ such that 
$
B(x,z,K,1)\leq CP(z)
$  for all $x\in\Z$.\\ \\
 Assumption \textit{(GM2)} is a natural sufficient assumption 
for the existence of the process and graphical construction below. 
The shift operator  $\tau_y$  on $\mathbf A$ is defined by
$
(\tau_y B)(x,z,n,m)=B(x+y,z,n,m)
$
We generalize \eqref{generator} by setting
\be\label{generator_genmis}
L_B f(\eta)=\sum_{x,y\in{\Z}}B(x,y-x,\eta(x),\eta(y)) \left[ f\left(\eta^{x,y} \right)-f(\eta)
\right]
\ee
Thus, the environment at site $x$ is given here by the jump rate function  $B(x,.,.)$  
with which jump rates from site $x$ are computed.\\ \\
For $v=(z,u)$, set  $m(dv)=CP(dz)\lambda_{[0,1]}(du)$
in \eqref{special_choice}, and replace \eqref{update_misanthrope} with
\be\label{update_genmis}
{\mathcal T}^{B,x,v}\eta=\left\{
\ba{lll}
\eta^{x,x+z} & \mbox{if} & \displaystyle{ u<\frac{B(x,z,\eta(x),\eta(x+z))}
{CP(z)}}\\
\eta &  & \mbox{otherwise}
\ea
\right.
\ee
 A natural irreducibility assumption generalizing 
\eqref{irreducibility_misanthrope} is the existence of a 
constant $C>0$ and a probability measure $p(.)$ on $\Z$ 
satisfying \eqref{irreducibility_misanthrope}, such that
\be\label{irreducibility_genmis}
\forall z\in\Z,\quad\inf_{x\in\Z}b(x,z,1,K-1)\geq c p(z)
\ee
The basic model \eqref{generator} is  recovered for
$B(x,z,n,m)=\alpha(x)p(z)b(n,m)$. Another natural example
is the Misanthrope's process with bond disorder.
Here ${\mathbf A}=[c,1/c]^{\Z^2}$ with the space shift defined by 
$\tau_z\alpha=\alpha(.+z,.+z)$. We set
$B(x,z,n,m)=\alpha(x,x+z)b(n,m)$,  where  $\alpha\in\mathbf A$. 
Assumption {\em (GM2)} is now equivalent to existence of a constant 
$C>0$ and a probability measure  $P(.)$  on $\Z$ such that 
$\alpha(x,y)\leq C P(y-x)$. \\ \\
 The microscopic flux function $j_2$ in \eqref{def_f} is given here by
\be\label{microflux_genmis}
j_2(\alpha,\eta)=\sum_{z\in\Z}zB(0,z,\eta(0),\eta(z))
\ee  
\textbf{Asymmetric exclusion process with overtaking.}
 This example is a particular case of the generalized $k$-step $K$-exclusion
studied in \cite[Section 5.2, Example 5.4]{bgrs4}, see also the traffic flow model in
\cite[Section 5.3]{bgrs4}.  The former model is itself a special case of the $k$-step misanthrope process
defined below.\\ 
Let $K=1$, $k\in\N$, and  $\mathfrak K$  denote the set of 
$(2k)$-tuples $(\beta^j)_{j\in\{-k,\ldots,k\}\setminus\{0\}}$ such that
$\beta^{j+1}\leq\beta^j$ for every $j=1,\ldots,k-1$, $\beta^{j-1}\leq\beta^j$ 
for every $j=-1,\ldots,-k+1$, and $\beta^1+\beta^{-1}>0$.
We define  ${\mathbf A}={\mathfrak K}^\Z$.  An element of $\mathbf A$ is denoted by 
$\beta=(\beta_x^j)_{j\in\{-k,\ldots,k\},x\in\Z}$.
The dynamics of this model is defined informally as follows. 
A site $x\in\Z$ is chosen as the initial site, 
then a jump direction (right or left) is chosen, and in this direction, 
the particle jumps to the first available site if it is no more than $k$ sites ahead.
The jump occurs at rate $\beta_x^j$ if the first available site is $x+j$. 
Let  $\mathcal V=[0,1]\times\{-1,1\}$ and $m=\delta_{1}+\delta_{-1}$.
For $x\in\Z$ and $v\in\{-1,1\}$, we set 
\[
N(x,v,\eta):=\inf\left\{
i\in\{1,\ldots,k\}:\,\eta(x+iv)=0
\right\}
\]
with the usual convention $\inf\emptyset=+\infty$.
The corresponding monotone transformation is defined  for $(u,v)\in\mathcal V$ by
\be\label{transfo_over}
{\mathcal T}^{\alpha,x,v}\eta=\left\{
\ba{lll}
\eta^{x,x+N(x,v,\eta)v} & \mbox{if  }\quad  N(x,v,\eta)<+\infty\mbox{ and }u\leq\beta_x^{vN(x,v,\eta)}\\ 
\eta & \quad\mbox{otherwise} 
\ea
\right.
\ee
 Monotonicity of the transformation $\mathcal T^{\alpha,x,z}$ is given by 
\cite[Lemma 5.1]{bgrs4},  and is also a particular case of Lemma \ref{attractive_kstep} below, 
 which states the same property for the $k$-step misanthrope's process.  
It follows from \eqref{gengen} and \eqref{transfo_over} 
that the generator of this process is given for $\beta\in\mathbf A$ by
\begin{eqnarray}\nonumber
L_\beta f(\eta) & = & \sum_{x\in\Z}\eta(x)\sum_{j=1}^k\left\{
\beta^j_x[1-\eta(x+j)]\prod_{i=1}^{i=j-1}\eta(x+i)\right.\\
& + & 
\left.\beta^{-j}_x[1-\eta(x-j)]\prod_{i=1}^{i=j-1}\eta(x-i)
\right\}\label{generator_overtaking}
\end{eqnarray}
A sufficient irreducibility property replacing 
\eqref{irreducibility_misanthrope} is the existence of a constant $c>0$ such that 
\be\label{irreducibility_overtaking}
\inf_{x\in\Z}\left(
\beta_x^1+\beta_x^{-1}
\right)>0
\ee 
The microscopic flux function $j_2$ in \eqref{def_f} is given here by
\begin{eqnarray}\nonumber
j_2(\beta,\eta) & = & \eta(0)\sum_{j=1}^k j\beta^j_0[1-\eta(j)]\prod_{i=1}^{j-1}\eta(i)\\
& - & \eta(0)\sum_{j=1}^k j\beta^{-j}_0[1-\eta(j)]\prod_{i=1}^{j-1}\eta(i)\label{flux_overtaking}
\end{eqnarray}
with the convention that an empty product is equal to $1$. 
For $\rho\in[0,1]$, let ${\mathcal B}_\rho$ denote the Bernoulli distribution 
on $\{0,1\}$. In the absence of disorder, that is when $\beta^i_x$ does not 
depend on $x$, the measure $\nu_\rho$ defined by
\[
\nu_\rho(d\eta)=\bigotimes_{x\in\Z}{\mathcal B}_\rho[d\eta(x)]
\]
is invariant for this process.  It follows from \eqref{flux_overtaking}  
that the macroscopic flux function for the model without disorder is given by
\be\label{macroflux_overtaking}
G(u)=(1-u)\sum_{j=1}^k j[\beta^j-\beta^{-j}]u^j
\ee\\ 
%
%
 \textbf{The $k$-step misanthrope's process.} 
%
In the sequel, an element of $\Z^k$ is denoted by
$\underline{z}=(z_1,\ldots,z_k)$.
Let  $K\geq 1$,  $k\geq 1$, $c\in(0,1)$. \\ \\
Define $\mathcal D_0$ to be the set of functions 
$b:\{0,\ldots,K\}^2\to\R^+$ such that $b(0,.)=b(.,K)=0$, 
$b(n,m)>0$ for $n>0$ and $m<K$, and $b$ is nondecreasing 
(resp. nonincreasing) w.r.t. its first (resp. second) argument.
Let $\mathcal D$ denote the set of functions
$b=(b^1,\ldots,b^k)$ from
$\Z^k\times\{0,\ldots,K\}^2\to(\R^+)^k$ such that $b^j(\underline{z},.,.)\in\mathcal D_0$
for each $j=1,\ldots,k$, and
%
\be\label{decreasing_rates} \forall j=2,\ldots,k, \quad b^j(.,K,0)\leq b^{j-1}(.,1,K-1)\ee
Let $q$ be a probability distribution on $\Z^k$, and $b\in\mathcal D$.
We define the $(q,b)$ $k$-step  misanthrope  
process as follows.
A particle at $x$  (if some)  picks a $q$-distributed random vector $\underline{Z}=(Z_1,\ldots,Z_k)$,
and jumps to the  first  site $x+Z_i$ ($i\in\{1,\ldots,k\})$  with strictly less than $K$ particles
along the path $(x+Z_1,\ldots,x+Z_k)$, if such a site exists, 
with rate $b^i(\underline{Z},\eta(x),\eta(x+Z_i))$.
Otherwise, it stays at $x$.\\ \\
Next, disorder is introduced:
the environment is a field
$\alpha=((q_x,b_x):\,x\in\Z)\in{\mathbf A}:=(\mathcal
P(\Z^k)\times\mathcal D)^\Z$. For a given realization of the
environment,
the distribution of the path $\underline{Z}$ picked by a particle
at $x$ is $q_x$, and the rate at which it jumps to $x+Z_i$ is $b^i_x(\underline{Z},\eta(x),\eta(x+Z_i))$.
The corresponding generator is given by
\be\label{christophe_kstep}
{L}_{\alpha}f(\eta )
 =  \sum_{i=1}^{k}\sum_{x,y\in\Z}c_\alpha^i(x,y,\eta )\left[
f(\eta
^{x,y})-f(\eta )\right]\label{eq:kstepgenerator}\ee
 for a local function $f$ on $\mathbf X$,
where (with the convention that an empty product is equal to $1$)
\[c_\alpha^i(x,y,\eta )=
\int\left[
b^i_x(\underline{z},\eta(x),\eta(y)) {\bf 1}_{\{x+z_i=y\}} \prod_{j=1}^{i-1}{\bf 1}_{\{\eta(x+z_j)=K\}}
\right]\,dq_x(\underline{z})
\]
The distribution $Q$ of the environment on $\mathbf A$  is assumed ergodic
with respect to the space shift  $\tau_y$, where
$\tau_y\alpha=((q_{x+y},b_{x+y}):\,x\in\Z)$. \\ \\
 A sufficient condition for the existence of the process and graphical construction below is
the existence of a probability measure $P(.)$ on $\Z$ and a constant $C>0$ such that
%
\be\label{summability_kstep}\sup_{i=1,...,k}\sup_{x\in\Z}q^i_x(.)  \leq   C^{-1}P(.)\ee
where $q_x^i$ denotes the $i$-th marginal of $q_x$.
On the other hand, a natural irreducibility assumption 
sufficient for Proposition \ref{invariant} and Theorem \ref{th:hydro}, is the existence of a constant $c>0$, and  a probability measure $p(.)$ on $\Z$ satisfying \eqref{irreducibility_kstep}, such that
\be \inf_{x\in\Z} q^1_x(.)  \geq  c p(.)\label{irreducibility_kstep}\ee\\
To define a graphical construction, we set, for
$(x,\underline{z},\eta)\in\Z\times\Z^k\times{\mathbf X}$,  $b\in\mathcal D_0$ and $u\in[0,1]$,
\beq\label{nbsteps} N(x,\underline{z},\eta) &=&
\inf\left\{i\in\{1,\ldots,k\}:\,\eta\left(x+z_i\right)<K\right\}
 \mbox{ with} \inf\emptyset=+\infty \\ 
\label{finloc} Y(x,\underline{z},\eta) &=& \left\{
\ba{lll}
x+z_{N(x,\underline{z},\eta)} & \mbox{if} & N(x,\underline{z},\eta)<+\infty\\
x & \mbox{if} & N(x,\underline{z},\eta)=+\infty
\ea
\right. \\
\label{update_kstep}
{{\mathcal T}_0}^{x,\underline{z},b,u}\eta&=&\left\{
\ba{lll}
\eta^{x,Y(x,\underline{z},\eta)} & \mbox{if} 
& u< b^{N(x,\underline{z},\eta)}(\underline{z},\eta(x),\eta(Y(x,\underline{z},\eta))) \\
\eta &  & \mbox{otherwise}
\ea
\right.
\eeq
 Let $\mathcal V=[0,1]\times[0,1]$,
$m=\lambda_{[0,1]}\otimes\lambda_{[0,1]}$. For each  probability
distribution $q$ on $\Z^k$, there exists a mapping
$F_q:[0,1]\to\Z^k$ such that $F_q(V_1)$ has distribution $q$ if
$V_1$ is uniformly distributed on $[0,1]$. Then the transformation
$\mathcal T$ in \eqref{update_rule} is defined by  (with
$v=(v_1,v_2)$ and  $\alpha=((q_x,\beta_x):\,x\in\Z))$
\be\label{update_kstep}
{\mathcal T}^{\alpha,x,v}\eta={\mathcal T}_0^{x,F_{q_x}(v_1),b_x(F_{q_x}(v_1),.,.),v_2}\eta
\ee
 The definition of $j_2$ in \eqref{def_f}, applied to the generator \eqref{eq:kstepgenerator}, yields
\be\label{flux_kstepmis}
j_2(\alpha,\eta)=\sum_{z\in\Z}zc_\alpha(0,z,\eta)
\ee
where
\[
c_\alpha(x,y,\eta):=\sum_{i=1}^k c_\alpha^i(x,y,\eta)
\]  
\textbf{Special cases.} \\ \\
1. The generalized misanthrope's process is recovered for $k=1$, 
because then in \eqref{christophe_kstep} we have
$c_\alpha^1(x,y,\eta)=q_x^1(y-x)b_x^1(y-x,\eta(x),\eta(y))$. \\ \\
2.  A {\em generalized disordered} $k$-step  exclusion  process is obtained if $K=1$ 
and  $b^j_x(\underline{z},n,m)=\beta^j_x(\underline{z})n(1-m)$. In this process,
if site $x$ is the initial location of an attempted jump, 
and a particle is indeed present at $x$, a random path of length $k$ 
with distribution $q^i_x$ is picked, and the particle tries to find an empty location 
along this path. If it finds none, then it stays at $x$. Previous versions of the 
$k$-step exclusion process are recovered
if one makes special choices for the distribution $q_x^i$:\\ \\
2a. The usual $k$-step exclusion process with site disorder, whose generator 
was given by \eqref{generator_k}, corresponds to the case where $q_x^i$ is 
the distribution of the first $k$ steps of a random walk with kernel $p(.)$ 
absorbed at $0$, and $\beta^j_x=\alpha(x)$.\\ \\
2b. The exclusion process with overtaking, whose generator was given by 
\eqref{generator_overtaking}, corresponds to the case where the random path 
is chosen as follows: first, one picks with equal probability a jumping direction 
(left or right); next, one moves in this direction by successive deterministic jumps of size $1$.\\ \\
3. For $K\geq 2$, the generalized $k$-step $K$-exclusion process 
(\cite[Subsection 5.2]{bgrs4})  corresponds to
$b^j_x(\underline{z},n,m)=\beta^j_x(\underline{z}){\bf 1}_{\{n>0\}}{\bf 1}_{\{m<K\}}$.\\ \\
Returning to the general case, condition \eqref{decreasing_rates} is the relevant extension of the 
condition $\beta^j_x(\underline{z})\leq\beta^{j-1}_x(\underline{z})$ in the exclusion process with overtaking. 
If $K\geq 2$, it means that {\em any} possible $j$-step jump has rate 
larger or equal than {\em any} $(j-1)$-step jump.\\ \\
 The monotonicity property \eqref{attractive_1} of the graphical construction, 
and thus the complete monotonicity of the process, is a consequence of the following lemma.
\begin{lemma}\label{attractive_kstep}
For every $(x,\underline{z},u)\in\Z\times\Z^k\times[0,1]$,
${\mathcal T}_0^{x,\underline{z},\beta,u}$ is an increasing
mapping from ${\mathbf X}$ to ${\mathbf X}$.
\end{lemma}
\begin{proof}{lemma}{attractive_kstep}
Let $(\eta,\xi)\in{\mathbf X}^2$  with $\eta\leq\xi$. To prove that
${\mathcal T}_0^{x,\underline{z},\beta,u}\eta\leq {\mathcal
T}_0^{x,\underline{z},b,u}\xi$, since $\eta$ and $\xi$ can only
possibly change at sites $x$, $y:=Y(x,\underline{z},\eta)$  and
$y':=Y(x,\underline{z},\xi)$,  it is sufficient to verify the
inequality at these sites.
\\ \\
If $\xi(x)=0$, then by \eqref{update_kstep}, $\eta$ and $\xi$ are
both unchanged by ${\mathcal T}_0^{x,\underline{z},b,u}$. If
$\eta(x)=0<\xi(x)$, then  ${\mathcal
T}_0^{x,\underline{z},b,u}\xi(y')
\geq\xi(y')\geq\eta(y')={\mathcal T}_0^{x,\underline{z},b,u}\eta(y')$.  \\ \\
Now assume $\eta(x)>0$. Then $\eta\leq\xi$ implies 
$N(x,\underline{z},\eta)\leq N(x,\underline{z},\xi)$. If 
$N(x,\underline{z},\eta)=+\infty$,  $\eta$ and $\xi$ are unchanged.
If $N(x,\underline{z},\eta)<N(x,\underline{z},\xi)=+\infty$, 
then ${\mathcal T}_0^{x,\underline{z},b,u}\eta=\eta^{x,y}$, and  $\xi(y)=K$.  Thus,
${\mathcal T}_0^{x,\underline{z},b,u}\eta(x)=\eta(x)-1\leq\xi(x)
={\mathcal T}_0^{x,\underline{z},b,u}\xi(x)$, and   
${\mathcal T}_0^{x,\underline{z},b,u}\xi(y)=\xi(y)
=K\geq{\mathcal T}_0^{x,\underline{z},b,u}\eta(y)$. \\ \\
In the sequel, we assume $N(x,\underline{z},\eta)$ and 
$N(x,\underline{z},\xi)$ both finite. Let 
$\beta:=b^{N(x,\underline{z},\eta)}(\eta(x),\eta(y))$ 
and $\beta':=b^{N(x,\underline{z},\xi)}(\xi(x),\xi(y'))$. \\ \\
1) Assume $N(x,\underline{z},\eta)=N(x,\underline{z},\xi)<+\infty$, then $y=y'$.
If $u\geq\max(\beta,\beta')$, both $\eta$ and $\xi$ are unchanged. If $u<\min(\beta,\beta')$,
${\mathcal T}_0^{x,\underline{z},b,u}\eta=\eta^{x,y}$ and
${\mathcal T}_0^{x,\underline{z},b,u}\xi=\xi^{x,y}$, whence the
conclusion. We are left to examine different cases where 
$\min(\beta,\beta')\leq u<\max(\beta,\beta')$.\\ \\
a) If $\eta(x)=\xi(x)$, then $\beta'\leq \beta$, 
and $\beta'\leq u<\beta$ implies $\eta(y)<\xi(y)$. In this case,
${\mathcal T}_0^{x,\underline{z},b,u}\xi(x)=\xi(x)\geq\eta(x)>{\mathcal T}_0^{x,\underline{z},b,u}\eta(x)$ and
${\mathcal T}_0^{x,\underline{z},b,u}\xi(y)=\xi(y)=K\geq{\mathcal T}_0^{x,\underline{z},b,u}\eta(y)$.\\ \\
b) If $\eta(x)<\xi(x)$, then 
${\mathcal T}_0^{x,\underline{z},b,u}\xi(x)\geq\xi(x)-1\geq\eta(x)
\geq{\mathcal T}_0^{x,\underline{z},b,u}\eta(x)$.
If $\beta\leq u<\beta'$, then 
${\mathcal T}_0^{x,\underline{z},b,u}\xi(y)=\xi(y)+1>\eta(y)={\mathcal T}_0^{x,\underline{z},b,u}\eta(y)$.
If $\beta'\leq u<\beta$, then $\eta(y)<\xi(y)$ and 
${\mathcal T}_0^{x,\underline{z},b,u}\xi(y)=\xi(y)\geq\eta(y)+1
={\mathcal T}_0^{x,\underline{z},b,u}\eta(y)$.\\ \\
2) Assume
$N(x,\underline{z},\eta)<N(x,\underline{z},\xi)<+\infty$, hence
 $\beta\geq \beta'$ by \eqref{decreasing_rates}, and  $\eta(y)<\xi(y)=K$.
If $u\geq \beta$, $\eta$ and $\xi$ are unchanged.
 If $u<\beta'$, then  ${\mathcal
T}_0^{x,\underline{z},b,u}\eta(y)=\eta(y)+1\leq\xi(y)={\mathcal
T}_0^{x,\underline{z},b,u}\xi(y)$,  and ${\mathcal
T}_0^{x,\underline{z},b,u}\xi(y')=\xi(y')+1\geq {\mathcal
T}_0^{x,\underline{z},b,u}\eta(y')=\eta(y')$.  If $\beta'\leq u<\beta$,
then  ${\mathcal T}_0^{x,\underline{z},\beta,u}\eta(x)=\eta(x)-1\leq
{\mathcal T}_0^{x,\underline{z},b,u}\xi(x)$ and   ${\mathcal
T}_0^{x,\underline{z},b,u}\eta(y)=\eta(y)+1\leq {\mathcal
T}_0^{x,\underline{z},b,u}\xi(y)=\xi(y)=K$.
\end{proof}

\begin{thebibliography}{99}
%
\bibitem{andjel}
Andjel, E.D. Invariant measures for the zero range process. {\em
Ann. Probab.} {\bf 10} (1982), no. 3, 525--547.
%
 \bibitem{afs}
Andjel, E., Ferrari, P.A., Siqueira, A. Law of large numbers for the
asymmetric exclusion process. {\em Stoch. Process. Appl.} {\bf 132}
(2004) no. 2, 217--233. 
%
 \bibitem{ak}
Andjel, E.D., Kipnis, C. Derivation of the hydrodynamical equation
for the zero range interaction process. {\em Ann. Probab.} {\bf 12}
(1984), 325--334. 
%
\bibitem{av}
Andjel, E.D., Vares, M.E. Hydrodynamic equations for attractive
particle systems on $\Z$. {\em J. Stat. Phys.} {\bf 47} (1987), no.
1/2, 265--288.
%
Correction to : ``Hydrodynamic equations for attractive particle
systems on $\Z$''. {\em J. Stat. Phys.} {\bf 113} (2003), no. 1-2, 379--380.
%
\bibitem{ba} Bahadoran, C.
Blockage hydrodynamics of driven conservative
systems. {\em Ann. Probab.} {\bf 32}  (2004),  no. 1B, 805--854.
%
\bibitem{bgrs1}
Bahadoran, C., Guiol, H., Ravishankar, K., Saada, E. A constructive
approach to Euler hydrodynamics for attractive particle systems.
Application to $k$-step exclusion. {\em Stoch. Process. Appl.} {\bf
99} (2002), no. 1, 1--30.
%
\bibitem{bgrs2}
Bahadoran, C., Guiol, H., Ravishankar, K., Saada, E. Euler
hydrodynamics of one-dimensional attractive particle systems. {\em
Ann. Probab.} {\bf 34} (2006), no. 4, 1339--1369.
%
\bibitem{bgrs3}
Bahadoran, C., Guiol, H., Ravishankar, K., Saada, E. Strong
hydrodynamic limit for attractive particle systems on $\Z$. {\em
Elect. J. Probab.} {\bf 15} (2010), no. 1, 1--43.
%
\bibitem{bgrs4}  Bahadoran, C.; Guiol, H.; Ravishankar, K.; Saada, E.
Euler hydrodynamics for attractive particle systems in random environment.
\emph{Ann. Inst. H. Poincar\'e Probab. Statist.} \textbf{50} (2014), no. 2, 403--424. 
%
\bibitem{bmrs1} Bahadoran, C.; Mountford, T.S.; Ravishankar, K.; Saada, E. 
Supercriticality conditions for the asymmetric zero-range process with sitewise disorder. 
\emph{Braz. J. Probab.  Stat.} {\bf 29}, no. 2 (2015), 313--335.
%
\bibitem{bmrs2} Bahadoran, C.; Mountford, T.S.; Ravishankar, K.; Saada, E. 
Supercritical behavior of zero-range process with sitewise disorder. To appear in 
\emph{Ann. Inst. H. Poincar\'e Probab. Statist.} (2016).
%
\bibitem{bmrs3} Bahadoran, C.; Mountford, T.S.; Ravishankar, K.; Saada, E.:
Supercritical hydrodynamics for asymmetric zero-range process with site disorder. In progress.
\bibitem{bmrs4} Bahadoran, C.; Mountford, T.S.; Ravishankar, K.; Saada, E. 
Convergence and quenched local equilibrium for asymmetric zero-range process 
with site disorder. In progress.
%
 \bibitem{ballou} Ballou, D.P.  Solutions to nonlinear hyperbolic
Cauchy problems without convexity conditions. {\em Trans. Amer.
Math. Soc.} {\bf 152} (1970), 441--460. 
%
 \bibitem{bf}
Benassi, A., Fouque, J.P. Hydrodynamical limit for the asymmetric
exclusion process. {\em Ann. Probab.} {\bf 15} (1987), 546--560. 
%
\bibitem{bfl} Benjamini, I., Ferrari, P.A., Landim, C. Asymmetric
processes with random rates. {\em Stoch. Process. Appl.}  {\bf 61}
(1996), no. 2, 181--204.
%
\bibitem{bm}
Bramson, M., Mountford, T. Stationary blocking measures for
one-dimensional nonzero mean exclusion processes. {\em Ann. Probab.}
{\bf 30} (2002), no. 3, 1082--1130.
%
 \bibitem{bres}
Bressan, A.  {\em Hyperbolic systems of
conservation laws: the one-dimensional Cauchy problem.} Oxford
lecture series in Mathematics {\bf 20}, 2000. 
%
\bibitem{coc}
Cocozza-Thivent, C. Processus des misanthropes. {\em Z. Wahrsch.
Verw. Gebiete} {\bf 70} (1985), no. 4, 509--523.
%
\bibitem{dplm}
Dai Pra, P., Louis, P.Y., Minelli, I.
Realizable monotonicity for continuous-time Markov processes.
{\em Stoch. Process. Appl.} {\bf 120} (2010), no. 6, 959--982.
%
 \bibitem{dp} De Masi A., Presutti E. {\em Mathematical
Methods for Hydrodynamic Limits.} LN in Math. 1501, Springer, 1991. 
%
\bibitem{ev} Evans, M.R. Bose-Einstein condensation in
disordered exclusion models and relation to traffic flow.
{\em Europhys. Lett.} {\bf 36} (1996), no. 1, 13--18.
%
%
 \bibitem{fgs} Fajfrov\`a, L., Gobron, T., Saada, E. 
Invariant measures of Mass Migration Processes. 
\textit{Electron. J. Probab.}, {\bf 21}, no. 60 (2016), 1--52.  
%
\bibitem{fm}
Fill, J.A., Machida, M. Stochastic monotonicity and realizable monotonicity.
{\em Ann. Probab.} {\bf 29} (2001), no. 2, 938--978.
%
 \bibitem{frigray} 
Fristedt, B.; Gray, L. {\em A modern approach to probability theory.}
 Probability and its Applications. Birkh\"auser Boston, 1997. 
 %
 \bibitem{glimm}
Glimm, J.  Solutions in the large for nonlinear
hyperbolic systems of equations. Comm. Pure Appl. Math. {\bf 18} (1965),
697--715. 
%
 \bibitem{gs} Gobron, T., Saada, E. Couplings,
attractiveness and hydrodynamics for conservative particle systems.
{\em Ann. Inst. H. Poincar\'e Probab. Statist.} {\bf 46} (2010), no. 4, 1132--1177.
%
 \bibitem{gr}
Godlewski, E., Raviart, P.A. {\em Hyperbolic
systems of conservation laws}. Math\'ematiques \& Applications.
Ellipses, 1991.  
%
\bibitem{guiol}
Guiol, H. Some properties of $k$-step exclusion processes. {\em J.
Stat. Phys.} {\bf 94} (1999), no. 3-4, 495--511.
%
 \bibitem{h}
Harris, T.E. Nearest neighbor Markov interaction processes on
multidimensional lattices. {\em Adv. in Math.} {\bf 9} (1972),
66--89.
%
\bibitem{har}
Harris, T.E. Additive set-valued Markov processes and graphical
methods. {\em Ann. Probab.} {\bf 6} (1978), 355--378. 
%
\bibitem{kk} Kamae, T., Krengel, U. Stochastic partial ordering.
{\em Ann. Probab.} {\bf 6} (1978), no. 6, 1044--1049.
%
\bibitem{kl}
Kipnis, C., Landim, C. {\em Scaling limits of interacting particle
systems}. Grundlehren der Mathematischen Wissenschaften [Fundamental
Principles of Mathematical Sciences], {\bf 320}. Springer-Verlag,
Berlin, 1999.
%
 \bibitem{k}
Kru\v{z}kov, S. N.  First order quasilinear
equations with several independent variables.  {\em Math. URSS
Sb.} {\bf 10} (1970), 217--243.
%
 \bibitem{lig}
Liggett, T.M. Coupling the simple exclusion process. {\em Ann.
Probab.} {\bf 4} (1976), no. 3, 339--356.
%
\bibitem{lig1}
Liggett, T.M. {\em Interacting particle systems.} Classics in
Mathematics (Reprint of first edition), Springer-Verlag, New York,
2005.
%
 \bibitem{ls} Liggett, T. M., Spitzer, F.  
Ergodic theorems for coupled random walks and other systems with 
locally interacting components. {\em Z. Wahrsch. Verw. Gebiete}
\textbf{56} (1981), no. 4, 443--468. 
%
\bibitem{mrs}  Mountford, T.S., Ravishankar, K., Saada, E.
 Macroscopic stability for nonfinite range kernels.
 {\em  Braz. J. Probab. Stat.} {\bf 24} (2010), no. 2, 337--360.
%
\bibitem{fraydoun} Rezakhanlou, F. Hydrodynamic limit for attractive particle
systems on $ \Z\sp d$. {\em Comm. Math. Phys.} {\bf 140} (1991), no.
3, 417--448.
%
 \bibitem{rez}
Rezakhanlou, F.  Continuum limit for some growth
models. II. {\em Ann. Probab.} {\bf 29} (2001), no. 3, 1329--1372. 
%
\bibitem{rost} Rost, H. Nonequilibrium behaviour of a many
particle process: density profile and local equilibria, {\em Z.
Wahrsch. Verw. Gebiete} {\bf 58} (1981) no. 1, 41--53. 
%
\bibitem{ks} Sepp\"al\"ainen, T., Krug, J.
Hydrodynamics and Platoon formation for a totally asymmetric
exclusion model with particlewise disorder. {\em J. Stat. Phys.}
{\bf 95} (1999), no. 3-4, 525--567.
%
 \bibitem{timo}
Sepp\"al\"ainen, T. Existence of hydrodynamics for the totally
asymmetric simple $K$-exclusion process. {\em Ann. Probab.} {\bf 27}
(1999), no. 1, 361--415.
%
\bibitem{serre}
Serre, D. {\em Systems of conservation laws. 1. Hyperbolicity,
entropies, shock waves}. Translated from the 1996 French original by I.
N. Sneddon. Cambridge University Press, Cambridge, 1999.
%
 \bibitem{spohn} Spohn, H. {\em Large Scale Dynamics of
Interacting Particles.} Springer,  1991. 
%
\bibitem{strassen}
Strassen, V. The existence of probability measures with given
marginals. {\em Ann. Math. Statist.} {\bf 36} (1965), no. 2, 423--439.
%
 \bibitem{swart} Swart, J.M.  A course in interacting particle systems. Preprint. 
http:staff.utia.cas.cz/swart/lecture\_notes/partic15\_1.pdf   
%
 \bibitem{vol}
Vol'pert, A.I.  The spaces ${\rm BV}$ and
quasilinear equations. {\em Math. USSR Sbornik} {\bf 2} (1967), no. 2,
225--267. 
%
\end{thebibliography}
\end{document}